\documentclass{article}  
\usepackage[pdftex,
bookmarksnumbered,
bookmarksopen,
colorlinks,
citecolor=blue,
linkcolor=blue,]{hyperref}
\usepackage{amsmath}
\usepackage{amssymb}
\usepackage{mathrsfs}
\usepackage{amsfonts}
\usepackage{amsthm}
\usepackage{color}
\usepackage{booktabs}
\usepackage{listings}
\usepackage{boxedminipage}
\usepackage{stmaryrd}
\usepackage{cite}
\usepackage{relsize}
\usepackage{lscape}
\usepackage[top=1.3in, bottom=1.5in, left=1.5in, right=1.5in]{geometry}
\usepackage{stmaryrd} 
\usepackage{relsize}%
\usepackage{url}
\usepackage{multirow}
\usepackage{makecell}
\usepackage{arydshln}  

\newtheorem{theorem}{Theorem}[section]

\newtheorem{proposition}{Proposition}[section]

\newtheorem{lemma}{Lemma}[section]
\newtheorem{corollary}{Corollary}[section]
\newtheorem{remark}{Remark}[section]

\newtheorem{fact}{Fact} 
\newtheorem{setting}{Setting} 

\newenvironment{mytabular}{\bgroup\tiny\tabular}{\endtabular\egroup}

          \newcommand{\st}[1]{\ensuremath{ {\rm st}\left(   #1\right)  }}
          \newcommand{\orth}[1]{\ensuremath{ {\rm orth}\left(   #1\right)  }}

\newcommand{\bigxiaokuohao}[1]{\ensuremath{ \left(  #1 \right) }}      
\newcommand{\bigjueduizhi}[1]{\ensuremath{ \left|  #1 \right| }}   
\newcommand{\bigdakuohao}[1]{\ensuremath{ \left\{  #1 \right\} }}         
\newcommand{\bigzhongkuohao}[1]{\ensuremath{ \left[   #1 \right] }}      

\newcommand{\xiaokuohao}[1]{\ensuremath{  (  #1  ) }}      

\newcommand{\bigonenorm}[1]{\ensuremath{ \left\|   #1 \right\|_1 }}    
\newcommand{\bigfnorm}[1]{\ensuremath{ \left\|   #1 \right\|_F }}    
    
\newcommand{\bignorm}[1]{\ensuremath{ \left\|   #1 \right\|  }}

\newcommand{\bigfnormsquare}[1]{\ensuremath{ \left\|   #1 \right\|_F^2 }}    
\newcommand{\biginfnorm}[1]{\ensuremath{ \left\|   #1 \right\|_{\infty} }}

\newcommand{\innerprod}[2]{\ensuremath{ \left\langle   #1 , #2\right\rangle }}      
\newcommand{\innerprodnotleftright}[2]{\ensuremath{  \langle   #1 , #2 \rangle }}

               \newcommand{\mbr}[1]{\ensuremath{\mathbb R^{#1}   }} 
                
             \newcommand{\SM}[1]{\ensuremath{\mathbb S^{#1}   }}    
                \newcommand{\PSD}[1]{\ensuremath{\mathbb S^{#1}_+   }}     
                \newcommand{\PSDD}[1]{\ensuremath{\mathbb S^{#1}_{++}   }}

                \newcommand{\rank}[1]{\ensuremath{ {\rm rank}(#1)   }}

\newcommand{\dataMat}{\ensuremath{     X } }  
\newcommand{\projMat}{\ensuremath{     U } }   
\newcommand{\dataMatTrans}{\ensuremath{     X^{\top} } }  
\newcommand{\subgradMat}{\ensuremath{     S } }   
\newcommand{\subGMat}{\ensuremath{     S } }  

	\definecolor{darkgray}{rgb}{0.66, 0.66, 0.66}

\title{On   Finite-Step Convergence of the Non-Greedy   Algorithm and Proximal Alternating Minimization Method with Extrapolation  for $L_1$-Norm PCA
}

\author{Yuning Yang\thanks{College of Mathematics and Information Science, Guangxi University, Nanning, 530004, China  (yyang@gxu.edu.cn).} \thanks{Center for Applied Mathematics of Guangxi, Guangxi University, Nanning, 530004, China.}                             
}

\begin{document} 
\maketitle

\begin{abstract}
	The classical non-greedy algorithm (NGA) \cite{nie2011robust}      and the recently proposed  proximal  alternating minimization method with extrapolation (PAMe)  \cite{wang2021linear}  for $L_1$-norm PCA are revisited and their finite-step convergence      are studied. 
	It is first shown that NGA can be interpreted as a conditional subgradient   or an alternating maximization method. By recognizing it as a conditional subgradient,  we prove that the iterative points generated by the algorithm will be constant in finitely many steps under a certain full-rank assumption; such an assumption can be removed when the projection dimension is one. 
	By treating the algorithm as an alternating maximization, we then prove that  the objective value   will be fixed after at most $\left\lceil   \frac{F^{\max}}{\tau_0} \right\rceil$ steps,  where the stopping point   satisfies certain optimality conditions.
	Then, a  slight modification of NGA with improved convergence properties is analyzed. It is shown that  the iterative points generated by the modified algorithm will not change after at most  $\left\lceil   \frac{2F^{\max}}{\tau} \right\rceil$ steps;  furthermore,    the stopping point   satisfies certain optimality conditions if the proximal parameter $\tau$ is small enough. 
	
	For PAMe,  it is proved that the sign variable will remain constant after finitely many steps  and the algorithm can output a point satisfying certain optimality condition, if the parameters are small enough and a full rank assumption is satisfied. Moreover, if there is no proximal term on the projection matrix related subproblem, then the iterative points generated by this modified algorithm will not change after at most $\left\lceil   \frac{4F^{\max}}{\tau(1-\gamma)} \right\rceil$ steps and the stopping point also satisfies certain optimality conditions, provided similar assumptions as those for PAMe. The full rank assumption can be removed when the projection dimension is one.

\end{abstract}

\section{Introduction}
In   the big data era, to deal with   data in high dimensional space, a commonly used preprocess tool is dimension reduction.    It is well known that Principal Component Analysis (PCA) is one of the most popular techniques for dimension reduction. However, a main drawback of the traditional PCA is its non-robustness to outliers,  due to that it is essentially a least-square loss ($L_2$-norm) based model, making it not effective in the presence of non-Guassian noise. 

To alleviate the   drawback of PCA, several new models have been proposed and studied, one of which in recent years is the $L_1$-norm PCA; see, e.g, \cite{kwak2008principal,nie2011robust,markopoulos2017efficient,kamrani2020reduced,wang2021linear} and the references therein. Roughly speacking, PCA maximizes the variance (in the least-square criterion) of the projection of the data points onto a new latent axis system, while $L_1$-norm PCA replaces the least-square criterion by the least obsolute one ($L_1$-norm) . Compared with  PCA,  the least absolute loss  employed by $L_1$-norm PCA   is less sensitive to heavy-tailed noise or outliers.  Besides $L_1$-norm PCA, another popular approach of robustifying PCA is to minimize the absolute subspace representation error; see, e.g,  \cite{ke2005robust,eriksson2010efficient,yu2012anefficient,tsagkarakis2016l1}; this approach will not be studied in this work.   Although $L_1$-norm PCA is more robust than  PCA, 
unlike PCA which can be solved by singular value decomposition (SVD), $L_1$-norm PCA does not admit a closed-form solution. On the other hand, finding an optimal solution to $L_1$-norm PCA is NP-hard in general \cite{mccoy2011two,markopoulos2014optimal}. 

Given the above understanding, several works have devoted their efforts to designing algorithms to resolve $L_1$-norm PCA. Early works used heuristic algorithms \cite{croux2007algorithms,choulakian2006l1norm} that have no   theoretical guarantee. \cite{kwak2008principal} first proposed an iterative algorithm when the projection dimension   is one (in what follows, we use $K$ to denote this dimension), and then use a greedy method to find the projection matrix for the $K> 1$ cases.   \cite{nie2011robust} proposed a non-greedy algorithm for the $K\geq 1$ cases that can simultaneously update each column of the (partial) projection matrix.  When $K=1$, it reduces to the iterative algorithm in \cite{kwak2008principal}.    \cite{nie2011robust} studied certain convergence properties of the non-greedy algorithm, while the convergence results are  limited. More precisely, only the monotonically increasing property of the objective value generated by the algorithm was strictly proved.  It will be shown later that the non-greedy algorithm of \cite{nie2011robust} can be equivalently written as a fixed-point iteration, which can be further understood as a conditional (sub)gradient or an alternating maximization method.
After the work of \cite{nie2011robust},   advanced methods have been proposed \cite{markopoulos2017efficient,markopoulos2014optimal,kim2020asimple,wang2021linear,wang2019globally}; just to name a few. An efficient    algorithm was proposed in \cite{markopoulos2017efficient,kundu2014fast} using bit-flipping; moreover, the algorithm was proved to stop in finitely many steps.   \cite{kim2020asimple} considered the $L_1$ kernel PCA model with $K=1$   and showed that the proposed algorithm converges in finitely many steps linearly. Very recently, \cite{wang2021linear} designed an alternating minimization method, called PAMe, which can be regarded as a proximal and extrapolated improvement of NGA. By showing that the Kurdyka-\L{}ojasiewicz exponent of the problem is $1/2$, the authors were able to prove that PAMe converges globally and linearly, and the output is   a critical point if a certain parameter condition is met. It was observed in \cite{wang2021linear} that PAMe is more efficient than NGA and some other algorithms.  Non-iterative approaches were studied in \cite{markopoulos2014optimal,mccoy2011two}, where \cite{mccoy2011two} considered polynomial-time approximation algorithms for the $K=1$ case, while the approach developed in \cite{markopoulos2014optimal} can find the global optimizer, which runs in polynomial-time if  the sample dimension is   fixed. 

With the above     advanced approaches, however, the early NGA of  \cite{nie2011robust} (and also the $K=1$ case   \cite{kwak2008principal})  is still valuable, as its idea  has been inherited by  methods for solving various   PCAs  such as  the $L_1$-norm tensor analysis \cite{markopoulos2018l1norm,chachlakis2020l1} and    $L_{21}$-norm PCA \cite{nie2021nongreedy}. However, as has been mentioned,  its convergence behavior is far from being studied. Therefore, this work intends to study the finite-step convergence of NGA and its improvement  PAMe. Specifically, the following results are obtained:

1.    NGA is first interpreted as a conditional subgradient and   its subsequential convergence is proved (Proposition \ref{prop:subsequential_conv}). Under a certain full rank assumption, the iterative points will be fixed within finitely many steps (Theorem \ref{th:global_conv}). The stopping point is a FOC point. The full rank   assumption has been removed   if $K=1$ (Theorem \ref{th:global_conv_K_1}). 
To   reduce the assumption, 
the algorithm  is then treated as an alternating maximization method. It is shown that after at most $\left\lceil   \frac{F^{\max}}{\tau_0} \right\rceil$ steps, the objective value will be fixed, and the algorithm finds   a FOC point  (Theorem \ref{th:finite_conv_original_alg_fixed_point_iteration}).    Here $F^{\max}$ denotes the global maximum of $L_1$-PCA, and $\tau_0$ will be specified in the related part.

2. To  further improve the convergence results   while avoiding assumptions,  we consider imposing a proximal term on the sign variable related subproblem (called $\subGMat$-PNGA in this work).   It is shown that after at most $\left\lceil   \frac{2F^{\max}}{\tau} \right\rceil$ steps, the iterative points of $\subGMat$-PNGA will not change anymore, and the stopping point is a FOC point provided a small enough parameter $\tau$ (Theorem \ref{th:finite_step_conv}) related to the proximal term. The reason    why    the proximal term is only imposed on the sign variable related subproblem  is also discussed.  

3. For PAMe,  the extrapolated parameter is small enough, then the sign variable will not change after finitely many steps; if in addition, the parameters related to the proximal terms are also small enough and a full rank assumption is met, then PAMe can output a FOC point in finitely many steps (Theorem \ref{th:finite_conv_pame}).  Furthermore, we show that, 
if the proximal term is only imposed to the sign variable (called $\subGMat$-PAMe in this work) and similar assumptions as those of PAMe hold, then after at most $\left\lceil   \frac{4F^{\max}}{\tau(1-\gamma)} \right\rceil$ steps, the iterative points generated by $\subGMat$-PAMe will not change anymore, and the stopping point is a FOC point (Theorem \ref{th:finite_conv_pame0}).   When $K=1$, the full rank assumption can   be removed  (Theorems \ref{th:finite_conv_pame_K1} and \ref{col:finite_conv_pame_K_1}).   

The  convergence results above are summarized in Table \ref{tab:convergence_results}. 
\begin{table*}[htbp]
	\renewcommand\arraystretch{1.5}
	\centering
	\caption{\label{tab:convergence_results}Finite-step convergence results obtained in this work. }
	{	\begin{mytabular}{ccccc}
			\toprule
			Algorithm		&   Proj. Dim.   &  Convergence results & Theorem   & Assumptions   \\
			\midrule
			\multirow{4}{*}{NGA \eqref{alg:fixed_point_iteration_l1_pca}}  &  $K\geq 1$  &  {Finite-step convergence of the iterative points; FOC point }  & Thm. \ref{th:global_conv} &  Full column rank of $\dataMat\subGMat^k$      \\
			\cdashline{2-5}[1pt/1pt]
			& $K=1$ &  The same as above & Thm. \ref{th:global_conv_K_1}     &  N/A \\
			\cmidrule{2-5}
			& $K\geq 1$ &  \makecell{Finite-step convergence of the objective value; FOC point\\ Upper bound of steps: $\left\lceil   \frac{F^{\max}}{\tau_0} \right\rceil$ }  &   Thm. \ref{th:finite_conv_original_alg_fixed_point_iteration} &   N/A \\
			\midrule 
			$\subGMat$-PNGA \eqref{alg:am_prox_l1_pca}    & $K\geq 1$  &   \makecell{Finite-step convergence of the iterative points; \\ Upper bound of steps: $\left\lceil   \frac{2F^{\max}}{\tau} \right\rceil$ }   & Thm. \ref{th:finite_step_conv} &  {N/A  (small $\tau$ to find FOC point)}     \\
			\midrule
			\multirow{3}{*}{PAMe \eqref{alg:pame}} &  {$K\geq 1$} &  \makecell{Finite-step convergence of the sign variable}    &  Thm. \ref{th:finite_conv_pame}  &   \makecell{Small $\gamma$  (full column rank of $\dataMat\subGMat^{k^\prime}$ \\and  small $\tau,\beta$ to find FOC point) }       \\
			\cdashline{2-5}[1pt/1pt]
			&   $K=1$& The same as above & Thm. \ref{th:finite_conv_pame_K1} &  Small $\gamma$ (small $\tau,\beta$ to find FOC point)   \\
			\midrule
			\multirow{3}{*}{$\subGMat$-PAMe \eqref{alg:pame_beta0}} &  {$K\geq 1$} &  \makecell{Finite-step convergence of the iterative points; \\ Upper bound of steps: $\left\lceil   \frac{4F^{\max}}{\tau(1-\gamma)} \right\rceil$ }    &  Thm. \ref{th:finite_conv_pame0}  &   \makecell{Full column rank of $\dataMat\subGMat^k$ and \\small $\gamma$ (small $\tau$ to find FOC point) }       \\
			\cdashline{2-5}[1pt/1pt]
			&   $K=1$& The same as above & Thm. \ref{col:finite_conv_pame_K_1} &  small $\gamma$ (small $\tau$ to find FOC point)   \\		
			\bottomrule
		\end{mytabular}%
	}
	\label{tab:effect_tau3}%
\end{table*}%

Note that   \cite{wang2021linear} commented that ``\emph{the convergence rates of the two algorithms remain
	unknown}'', and ``\emph{ the algorithm based on BF iterations is guaranteed to converge in a finite
	number of steps, while that based on FP iterations is not known to possess such a property}''. In the context, ``\emph{the two algorithms}''  and ``\emph{that based on FP iterations}'' mean the algorithms of \cite{kwak2008principal,nie2011robust}.
Therefore, the finite-step convergence results summarized in item 1   give    an affirmative answer   to the   comments above.

Considering the  algorithms with the finite-step convergence property \cite{markopoulos2017efficient,kim2020asimple},  the algorithm of \cite{kim2020asimple} was only designed for the $K=1$ case; although the algorithm in \cite{markopoulos2017efficient} is applicable for either $K=1$ or $K>1$, it was only formally analyzed for the   $K=1$ case that  the stopping point satisfies certain optimality conditions.  On the other side, the obtained theoretical results not only show   that NGA of \cite{nie2011robust} (and the $K=1$ case \cite{kwak2008principal}) and PAMe of \cite{wang2021linear}, which are applicable for $K\geq 1$ cases, find a FOC point in finitely many steps, but we   also provide an explicit upper bound for the number of steps.

Another advantage of this work is that the analysis is elementary, which is essentially based on some simple observations and basic properties of the polar decomposition.

We also remark that   sparse PCA algorithms have also been interpreted as conditional (sub)gradient \cite{luss2013conditional},  and some of our definitions follow those of \cite{luss2013conditional}.  

The remainder is organized as follows. Section \ref{sec:pre} introduces basic definitions of polar decomposition with   properties,  describes the $L_1$-norm PCA model, and presents optimality conditions.  Section \ref{sec:non_greedy_alg_finite_step_conv} considers conditional (sub)gradient and alternating maximization interpretations for NGA,  and studies     finite-step convergence for NGA. Section \ref{sec:am_convergence} considers  a proximal version of NGA and PAMe, and   establishes the finite-step convergence results.    Section \ref{sec:conclusions} draws some conclusions.

\emph{Notation.} Vectors are represented in boldface lowercase $(\mathbf a,\mathbf b,\ldots)$ and matrices correspond to italic  capitals $(A,B,\ldots)$.   
$\langle\cdot,\cdot\rangle$ denotes the inner product of two vectors/matrices; $\|\cdot\| = \sqrt{\langle\cdot,\cdot\rangle }$ denotes the Euclidean   norm for a vector while $\|\cdot\|_F$ the Frobenius norm for a matrix. $\bigonenorm{\cdot}$ means the $L_1$-norm. $\|\cdot\|_2$ means the spectral norm. $\SM{n\times n}$,  $\PSD{n\times n}$, and $\PSDD{n\times n}$ respectively represent the cones of symmetric, symmetric positive semidefinite,  and symmetric positive definite matrices of size $n\times n$.    $(\cdot)^{\top}$ stands for the matrix transposition. $\lambda_{\min}(\cdot)$   denotes the smallest eigenvalue of a symmetric matrix and $\sigma^+_{\min}(\cdot)$ denotes the smallest positive singular value of a matrix.     $\st{m,n}:=\{ \projMat \in\mbr{m\times n} \mid \projMat^\top\projMat = I   \}$ , $m\geq n$ denotes the set of partially orthonormal matrices, i.e., the Stiefel manifold, while $\operatorname{orth}(n):=\{ \projMat\in\mathbb R^{n\times n}\mid \projMat^\top\projMat = \projMat\projMat^\top = I  \}$. 
$A_{ij}$ means the $(i,j)$-th entry of a matrix $A$. 

The sign function $\operatorname{sgn}(\cdot)$   is given as $\operatorname{sgn}(x) = x/|x|$ if $x\neq 0$ and $\operatorname{sgn}(x)=0$ if $x=0$.  $\operatorname{sgn}(\cdot)$ applies to a vector or a matrix entry-wisely.

\section{Preliminaries} \label{sec:pre}

\subsection{Polar decomposition} \label{sec:pd}
Polar decomposition is important in   algorithm design and analysis for $L_1$-norm PCA.   The existence of polar decomposition and its connections with SVD can be found in   classical materials; see, e.g., \cite{higham1986computing,horn1990matrix}. We summarize them in the following  results. 
\begin{theorem}[Polar decomposition] \label{th:polar_dec}
	Let $C\in\mathbb R^{m\times n}$, $m\geq n$. Then there exist a partially orthonormal	 $U\in\st{m, n}$ and a   symmetric positive semidefinite matrix $H\in \PSD{n\times n}$ such that
	\[ 
	C = UH,~U^\top U = I\in \mathbb R^{n\times n},
	\]
	where $H$ is uniquely determined. $(U,H)$ is called the polar decomposition of $C$.  Furthermore, if ${\rm rank}(C)=n$, then $H$ is symmetric positive definite and $U$ is also uniquely determined.
\end{theorem}

Throughout this work, we   write polar decomposition as PD for short, and we respectively call $U$ and $H$ in PD of $C$ the $U$-factor and the $H$-factor. 
\begin{proposition}[c.f. \cite{higham1986computing,horn1990matrix}]\label{prop:pd_connect_svd}
	Let $C = P\Sigma Q^{\top}$ be a compact SVD of $C\in\mbr{m\times n}$, $m\geq n$, where $P\in \st{m,n}$, $Q \in \orth{n   }$, and $\Sigma\in\mbr{n\times n}$ is diagonal with the singular values being nonnegative and arranged in a descending order. Then 
	$C=PQ^{\top}Q\Sigma Q^{\top} $, where $$U:=PQ^{\top} ~{\rm and}~ H:= Q\Sigma Q^{\top}$$ 
	give the PD of $C$. Conversely, if $C=UH$ is the PD of $C$, by writing down $H=Q\Sigma Q^{\top}$ as its spectral decomposition, then $(UQ)\Sigma Q^{\top}$ is a compact SVD of $H$. 
\end{proposition}

An important conclusion  about PD is given in the following, which can also be found in \cite{higham1986computing,horn1990matrix}. 
\begin{proposition}\label{prop:polar_max}
	Let  $(U,H)$ be a PD of $C\in\mathbb R^{m\times n}$, $m\geq n$ with $U\in\st{m, n}$ and $H\in\PSD{n\times n}$. Then 
	\begin{equation}\label{prob:key_subprob_prototype} 
	U\in\arg\max_{X\in \st{m,n}} \innerprod{C}{X}.
	\end{equation}
	In particular, if $\rank{C}=n$, then $U$ is uniquely determined and $H\in\PSDD{n\times n}$. 	Conversely, if $U$ is a maximizer of the above problem, then there exists an $H \in\PSD{n\times n}$ such that $C=UH$.
\end{proposition}

Note that the form of \eqref{prob:key_subprob_prototype} is involved as a key subproblem in   several $L_1$-PCA algorithms \cite{nie2011robust,wang2021linear,wang2019globally,zheng2022linearly,chachlakis2020l1,markopoulos2018l1norm,markopoulos2014optimal}. However, it seems that it is rarely mentioned in the literature that \eqref{prob:key_subprob_prototype} has connections with PD, except \cite{tsagkarakis2018l1_norm}.

The following error estimation is useful.  
\begin{lemma}
	\label{lem:polar_max_sufficient}
	Let  $C\in \mbr{m\times n}$, $m\geq n$ and $C=UH$ be a PD, where  
	$U\in\st{m, n},~H\in\PSD{n\times n}$. Then   there is a $\lambda\geq 0$ which is equal to the smallest eigenvalue of $H$, such that
	\[			     
	\innerprod{U}{C}- \innerprod{X}{C} \geq \frac{\lambda}{2}\| U - X\|_F^2, ~\forall X\in \st{m,n};
	\]
	in particular, if $\rank{C}=n$, then $\lambda >0$.
\end{lemma}
\begin{proof}
	We follow the notations in Proposition \ref{prop:pd_connect_svd} to write $H = Q\Lambda Q^{\top}$ be its spectral decomposition. We further denote $\Lambda := {\rm diag}[\lambda_1,\ldots,\lambda_n]$ with $\lambda_i\geq 0$, and $\sqrt H := Q{\rm diag}[\sqrt{ \lambda_1},\ldots,\sqrt{ \lambda_n}]Q^{\top}$. Let $\lambda:= \lambda_n = \lambda_{\min}(H)$. Then,
	\begin{align*}			 
	&\innerprod{C}{U} - \innerprod{C}{X} \\
	=&\innerprod{UH}{U} - \innerprod{UH}{X} \\
	=& \innerprod{H}{U^{\top}U} - \innerprod{UH}{X} \\
	=& \frac{1}{2}\left( \innerprod{H}{U^{\top}U} -2\innerprod{H}{U^{\top}X} + \innerprod{H}{X^{\top}X}   \right)\\
	=&\frac{1}{2}\bigfnorm{ \bigxiaokuohao{ U-X   } \sqrt H   }^2 \\
	\geq& \frac{\lambda}{2} \bigfnorm{ U-X  }^2,
	\end{align*}
	where the third equality is due to $\innerprod{H}{U^{\top}U} = \innerprod{H}{I} = \innerprod{H}{X^{\top}X}$. If $\rank{C}=n$ then $H\in \PSDD{n\times n}$ is positive definite and hence $\lambda = \lambda_n>0$. 
\end{proof}

\begin{remark}\label{rmk:pd_error_estimation}
	The above estimation also explains why \eqref{prob:key_subprob_prototype} has a unique solution when $\rank{C}=n$. 
	
	We  discuss a little more on the nonuniqueness of $\projMat$ when $C$ is not of full column rank. We still follow the notations in Proposition \ref{prop:pd_connect_svd}. Assume now that $\rank{C}=r<n<m$ and write $P = [\mathbf p_1,\ldots,\mathbf p_n]$. Let $\mathbf p^\prime$ be any normalized vector in the orthogonal complement of $P$ and denote $\tilde P :=[\mathbf p_1,\ldots,\mathbf p_{n-1},\mathbf p^\prime]$ and $\tilde \projMat := \tilde PQ^{\top}$. Then $\innerprod{C}{\projMat} - \innerprodnotleftright{C}{\tilde\projMat} = \innerprodnotleftright{P\Sigma Q^{\top}}{PQ^{\top}}  - \innerprodnotleftright{P\Sigma Q^{\top}}{   \tilde PQ^{\top}} = \innerprod{\Sigma}{I} - \innerprodnotleftright{\Sigma}{\tilde{P}^{\top}P} = \innerprod{\Sigma}{I} - \innerprodnotleftright{\Sigma}{\operatorname{diag}[1,\ldots,1,0]}=0$, where the last equality holds because $\Sigma_{nn}=0$. Thus $\tilde{U}$ is another $U$-factor of the PD of $C$.
\end{remark}

The following lemma is also useful.
\begin{lemma}
	\label{prop:relation:2}
	Given a nonzero $A \in\mathbb R^{m\times n}$, $m\geq n$, and let $B= \tau \projMat + A \neq 0$, where     $\tau>0$, $U \in \st{m,n}$. If $\projMat$ is a $\projMat$-factor of PD of $B$ and $\tau < \sigma^+_{\min}(A)$, then $\projMat$ is also a $\projMat$-factor of PD of $A$.
\end{lemma}
\begin{proof}
	Let   $H\in\mathbb S^{n\times n}_+$ such that $B = UH$ be a PD of $B$. In addition, denote a compact SVD of $B = P\Sigma_B Q^\top$, where $P \in \st{m,n}$, $Q\in \orth{n,n}$, and $\Sigma_B = {\rm diag}(\sigma_1(B),\ldots,\sigma_n(B) )$ with the singular values of $B$ satisfying $\sigma_1(B)\geq \cdots\geq \sigma_n(B)\geq 0$. Then Proposition \ref{prop:pd_connect_svd} shows that $H = Q\Sigma_B Q^\top$, and $U=PQ^\top$.  
	On the other hand, write $s:= [ {\rm sgn}(\sigma_1(B) -\tau),\ldots, {\rm sgn}(\sigma_n(B)-\tau) ]^\top\in \mathbb R^n$ and denote $S$ as a diagonal matrix with entries of $s$ on the diagonal elements of $S$ (in particular, we replace every $S_{ii}=0$ by $S_{ii}=1$ if this happens). 
	Then $A$ can be expressed as
	\begin{align}\label{eq:pd:1}
	\begin{split}
	A &= B-\tau \projMat = P(\Sigma_B - \tau I)Q^\top \\
	& = P \cdot{\rm diag}  \bigzhongkuohao{ |\sigma_1(B)-\tau|,\ldots,|\sigma_n(B)-\tau|}  \cdot \bigxiaokuohao{QS}^\top;
	\end{split}
	\end{align}
	it can be seen that	the above expression is a compact SVD of $A$, with singular values being $|\sigma_1(B)-\tau|,\ldots,|\sigma_n(B)-\tau|$  (not necessarily arranged in  the descending order). 
	
	We next show that if $\lambda < \sigma^+_{\min}(A)$, then ${\rm sgn}(\sigma_i(B)-\tau) \neq -1$ for all $i$, i.e., all the $\sigma_i(B)\geq \tau$.  Suppose on the contrary that there exists  an $  \hat i$ such that ${\rm sgn}(\sigma_{\hat i}(B)-\tau)=-1$, $1\leq \hat i\leq n$. Since $|\sigma_{\hat i}(B)-\tau|$ is a singular value of $A$, this together with ${\rm sgn}(\sigma_{\hat i}(B)-\tau)=-1$ means that $|\sigma_{\hat i}(B)-\tau|    $ is a nonzero singular value of $A$, and so  $	|\sigma_{\hat i}(B)-\tau|\geq \sigma^+_{\min}(A) >\tau$. This further means that
	$$
	\tau-\sigma_{\hat i}(B)>\tau\Leftrightarrow \sigma_{\hat i}(B)<0,$$  which contradicts   that $\sigma_i(B)\geq 0$ for all $i$. As a result,  ${\rm sgn}(\sigma_i(B)-\tau) \geq 0$ for all $i$,  i.e,   $ \hat H := Q(\Sigma_B-\tau I)Q^\top \in\mathbb S^{n\times n}_+$. This also shows that   $S=I$ (since we have replaced $0$ by $1$ on the diagonal entries of $S$). \eqref{eq:pd:1} together with the definitions of $\projMat$ and $\hat H$ gives that 
	$A = U\hat H$, i.e., $\projMat$ is a $\projMat$-factor of $A$.
\end{proof}

When $n=1$, PD has a simple form:
\begin{remark}\label{rmk:pd_n1}
	When $n=1$, i.e., $C$ is a column vector, $\projMat = \operatorname{PD}(C) = \frac{C}{\bignorm{C}}$, and now $H=\|C\|$. 
\end{remark}

In the $n=1$ case, we present an analogue of Lemma \ref{lem:polar_max_sufficient} for convenience.
\begin{corollary}
	\label{col:polar_max_sufficient_n1}
	Let $C\in\mathbb R^{m}$, with $U = \frac{C}{\|C\|}$ and $H=\|C\|$ being the PD of $C$. Then 
	\[
	\innerprod{U}{C} - \innerprod{X}{C} \geq \frac{\|C\|}{2}\bigfnorm{U-X}^2,~\forall  X\in\mathbb R^m ~{\rm with}~\|X\|=1.
	\] 
\end{corollary}

Before ending this subsection,  throughout this work, we will use the notation $[U,H  ] = \operatorname{PD}(C)$ to denote a PD of $C$. If only the   $U$-factor is required, we simply write it as $U = \operatorname{PD}(C)$. If context permits, ``$=$'' here means ``belongs to''.

\subsection{PCA, $L_1$-norm PCA, and optimality conditions for $L_1$-norm PCA}

Given a set of $n$ samples of dimension $d$: $\dataMat = [\mathbf x_1,\ldots,\mathbf x_n]\in\mbr{d\times n}$, variance matrix $S_\dataMat = \frac{1}{n}\dataMat\dataMat^\top$, and the (partial) projection matrix $\projMat \in \st{d,K} $ with $K \geq 1$ (usually $K\leq \min\{d,n \}$), PCA maximizes the variance matrix in the projection subspace, which amounts to solving 
\begin{align*}
\max_{\projMat}~\innerprod{S_{\dataMat}}{\projMat\projMat^\top}~{\rm s.t.}~\projMat \in\st{d,K}.
\end{align*}
Since $\innerprod{S_{\dataMat}}{\projMat\projMat^\top}$ is equivalent to $\bigfnorm{\dataMat^\top\projMat}^2/n$, it is often written as
\begin{align*}
\max_{\projMat}~\bigfnorm{ \dataMat^{\top}\projMat  } ~{\rm s.t.}~\projMat \in\st{d,K}.
\end{align*}

Due to the use of the least-square loss,   PCA is sensitive to non-Gaussian noise. 
A popular alternative is to replace the least-square loss with the least absolute loss. i.e., to replace the $L_2$-norm with the $L_1$-norm \cite{kwak2008principal,nie2011robust,markopoulos2017efficient,wang2021linear}, resulting into the following $L_1$-norm PCA model:
\begin{align}\label{prob:l1_pca}
&\max_{\projMat}~F(\projMat):=\bigonenorm{ \dataMat^{\top}\projMat  } ~{\rm s.t.}~\projMat \in\st{d,K}. 
\end{align} 

When the projection dimention $K=1$, the projection matrix $\projMat$ reduces to a projection vector $\mathbf u\in\mathbb R^d$, and the $L_1$-norm PCA is given by
\begin{align}\label{prob:l1_pca_K1}
\max_{\mathbf u}~F(\mathbf u):=\bigonenorm{ \dataMat^{\top}\mathbf u  } ~{\rm s.t.}~\mathbf u^\top\mathbf u = 1. 
\end{align}
This special case was studied in \cite{kwak2008principal,kim2020asimple}. 

Although $F$ in \eqref{prob:l1_pca} is nonsmooth, its KKT point can be written down as usual. By introducing a dual variable $\Lambda\in \SM{K\times K}$ which is symmetric, its Lagrangian function is given by 
\[L(\projMat,\Lambda) = F(U) - \innerprod{\Lambda}{\projMat^{\top}\projMat - I}.\]
Then $\projMat\in \st{d,K}$ is a KKT point of \eqref{prob:l1_pca} if there is a $\Lambda \in \mathbb S^{K\times K}$such  that 
\[ \projMat \Lambda \in \partial F(\projMat),\]
where $\partial f(\mathbf x)$ denotes the subdifferential of a convex function $f$ at $\mathbf x$ defined as
\begin{align}\label{eq:subdifferential}\partial f(\mathbf x) := \bigdakuohao{ \boldsymbol{\xi} \mid f(\mathbf y) \geq f(\mathbf x) + \innerprod{\boldsymbol{\xi}}{\mathbf y-\mathbf x},\forall \mathbf y   }.\end{align}
By calculus rule, 
$
\partial F(\projMat) = \dataMat \partial \bigonenorm{V}\Big |_{V = \dataMat^{\top}\projMat}.
$
Therefore,    $\projMat \in \st{d,K}$ is a KKT point of \eqref{prob:l1_pca} if  there is a $\Lambda\in \mathbb S^{K\times K}$ such that $\projMat\Lambda \in \dataMat\partial \bigonenorm{V}\Big |_{V = \dataMat^{\top}\projMat}$, i.e.,
\begin{align}
\label{eq:KKT}
{\rm (KKT)}~~W=\projMat \Lambda,~   W \in   \dataMat \partial \bigonenorm{V}\Big |_{V = \dataMat^{\top}\projMat}.
\end{align}

On the other hand, we follow    \cite{luss2013conditional} to say that a matrix $\projMat\in\st{d,K}$ satisfies the first-order optimality criteria (FOC) if 
\begin{align}\small
\label{eq:FOC}
\begin{split}
{\rm (FOC)}~~&\exists W\in \dataMat\partial \bigonenorm{V}\Big |_{V = \dataMat^{\top}\projMat},~{\rm s.t.}\\ &~\innerprod{W}{\projMat - Z} \geq 0,  ~ \forall Z\in\st{d,K}.
\end{split}
\end{align}
Note that \eqref{eq:FOC} is equivalent to that $\projMat$ is a maximizer of the following maximization problem:
\begin{align}
\label{prob:FOC_equivalent_max}
\exists W\in \dataMat \partial \bigonenorm{V} \Big |_{V = \dataMat^{\top}\projMat},~\projMat \in \arg\max_{Z\in \st{d,K} } \innerprod{W}{Z}.
\end{align}
By Proposition \ref{prop:polar_max}, the above means that $\projMat$ is a PD factor of $W$. More precisely, $\projMat$ is a FOC point if and only if there is a symmetric positive semidefinite matrix $H\in \PSD{K\times K}$ such that
\begin{align}
\label{eq:FOC_PD}
\begin{split}
&W = \projMat H, ~W\in\dataMat\partial\bigonenorm{V}\Big |_{V = \dataMat^{\top}\projMat},~{\rm where}\\
&~~~~\projMat\in\st{d,K},~H\in\PSD{K\times K},
\end{split}
\end{align}
or we can equivalently rewrite it in the following  form, using the notation presented in the last of Section \ref{sec:pd}:
\begin{align}
\label{eq:FOC_fpf}
\projMat = \operatorname{PD}\bigxiaokuohao{W  },~ W\in \dataMat \partial \bigonenorm{V}\Big |_{V=\dataMat^{\top}\projMat}. 
\end{align}
In view of the above derivations, since \eqref{eq:FOC_PD} requires that $H\in \PSD{K\times K}$, while $\Lambda\in \SM{K\times K}$ in \eqref{eq:KKT}, we conclude that:
\begin{proposition}
	FOC point   \eqref{eq:FOC} $=$ \eqref{prob:FOC_equivalent_max} $=$ \eqref{eq:FOC_PD} $=$ \eqref{eq:FOC_fpf} $\subseteq$  KKT point \eqref{eq:KKT}.
\end{proposition}

\section{The Non-Greedy Algorithm and Finite-Step Convergence} \label{sec:non_greedy_alg_finite_step_conv}

NGA proposed in \cite{nie2011robust} for solving $L_1$-norm PCA \eqref{prob:l1_pca} involves the following two steps in each iteration:
\begin{center}
	\begin{boxedminipage}{8.75cm}
		\begin{align*}
		&1. {~\rm Compute}~\alpha_i = \operatorname{sgn}\bigxiaokuohao{(U^k)^{\top}\mathbf x_i}~{\rm and~let}~M = \sum^n_{i=1}\mathbf x_i\alpha_i^{\top};\\
		&2. {~\rm Compute~the~compact~SVD~of~}M~{\rm as}~M=P\Sigma Q^{\top}\\&~~~{\rm and~let}~U^{k+1}=PQ^{\top}.  
		\end{align*}
	\end{boxedminipage}
\end{center}
That is, the algorithm iteratively updates the projection matrix $\projMat^k$ via SVD of $M\in \mathbb R^{d\times K}$, where every column of $M$ is a linear combination of the columns of the data matrix $\dataMat$.

Since $\alpha_i^{\top} = \operatorname{sgn}\bigxiaokuohao{\mathbf x_i^{\top}U^k}$,  we see that 
\[M=\sum^n_{i=1}\mathbf x_i\operatorname{sgn}\bigxiaokuohao{\mathbf x_i^{\top}U^k} = X\operatorname{sgn}\bigxiaokuohao{X^{\top}U^k}.\]  On the other hand, it follows from Proposition \ref{prop:pd_connect_svd} that step 2 is exactly computing a PD of $M$. Using the notation presented in the last of subsection \ref{sec:pd}, we may denote step 2 as $U^{k+1}=\operatorname{PD}(M)$. Thus we have:
\begin{fact}
	NGA  proposed in \cite{nie2011robust} can be  equivalently written as the following simple fixed-point format:
	\begin{center}
		\begin{boxedminipage}{8.75cm}
			\begin{align}
			\label{alg:fixed_point_iteration_l1_pca}
			{\rm (NGA)}~~~~		\projMat^{k+1} = \operatorname{PD}\bigxiaokuohao{ \dataMat \operatorname{sgn}\bigxiaokuohao{\dataMat^{\top}\projMat^k  }  }.
			\end{align} 
		\end{boxedminipage}
	\end{center}
\end{fact}

When $K=1$,   $\projMat$ reduces to a column vector $\mathbf u\in\mbr{d}$, and it follows from Remark \ref{rmk:pd_n1} that NGA \eqref{alg:fixed_point_iteration_l1_pca} reads as follows:
\begin{align}\label{alg:fixed_point_l1_pca_K_1}
\mathbf u^{k+1} = \dataMat\mathbf s^k/\bigfnorm{\dataMat\mathbf s^k},~{\rm where}~\mathbf s^k = \operatorname{sgn}\bigxiaokuohao{\dataMatTrans\mathbf u^k} \in\mbr{n}.
\end{align}

\cite{nie2011robust} showed that the algorithm NGA \eqref{alg:fixed_point_iteration_l1_pca} monotonically increases the objective function of \eqref{prob:l1_pca}. In the coming subsection, we will equivalently view NGA \eqref{alg:fixed_point_iteration_l1_pca} as a conditional subgradient method or an alternating maximization method. These two perspectives are important, as they allow us to establish the finite-step convergence conveniently. Convergence results will be presented in Subsection \ref{sec:finite_step_conv_fpi}.

\subsection{Two   perspectives of NGA} \label{sec:two_perspective}

\subsubsection{NGA \eqref{alg:fixed_point_iteration_l1_pca} as a conditional (sub)gradient}
We first show that NGA \eqref{alg:fixed_point_iteration_l1_pca} is in fact an instance of the conditional (sub)gradient method (CG for short, also known as the Frank-Wolfe method). CG was originally proposed in \cite{frank1956algorithm} for solving the convex problem $\min_{\mathbf x\in C} f(\mathbf x) $ with the iteration:
\begin{align*}
&\mathbf x^{k+1} = \alpha^k \mathbf z^k + (1-\alpha)\mathbf x^k,~\alpha^k \in (0,1],~{\rm where}\\
&~~~~~~~\mathbf z^{k+1} \in\arg\min_{\mathbf x\in C}\innerprod{\nabla f(\mathbf x^k)}{\mathbf z};
\end{align*}
here $f $ is smooth. Recent developments of CG can be found in the survey \cite{freund2016new}. When $C$ is nonconvex, CG is not directly applicable.  \cite{luss2013conditional} proposed a CG with unit step-size framework (we follow \cite{luss2013conditional} to call it CondGradU) for maximizing a (nonsmooth) convex function $f$ over a compact (possibly nonconvex) set $C$ via the following simple scheme:
\begin{align}
\label{alg:condg-u_prototype}
{\rm (CondGradU)}~~\mathbf x^{k+1} \in\arg\max_{\mathbf x\in C} \innerprod{  \boldsymbol{\xi}^k}{\mathbf x},~{\rm where}~  \boldsymbol{\xi}^k\in\partial f(\mathbf x^k).
\end{align}
If $f$ is smooth, then $\partial f(\mathbf x) = \bigdakuohao{\nabla f(\mathbf x)}$ and $\boldsymbol{\xi}^k=\nabla f(\mathbf x^k)$.

Now, recall $L_1$-norm PCA \eqref{prob:l1_pca}. Since $F(\projMat) = \bigonenorm{\dataMatTrans\projMat}$ is convex and $\st{d,K}$ is compact, we can apply CondGradU \eqref{alg:condg-u_prototype} to solve $L_1$-norm PCA, which leads to the   following scheme:
\begin{align*}
&\projMat^{k+1} \in\arg\max_{U\in \st{d,K}}\innerprod{\dataMat\subGMat^k}{\projMat},~{\rm where}\\
&~~~~~~~\dataMat\subGMat^k\in \partial F(\projMat^k) = \dataMat\partial \bigonenorm{V}\Big |_{V = \dataMatTrans \projMat^k},
\end{align*}
i.e., $\dataMat\subGMat^k$ is a subgradient of $F(\cdot)$ at $\projMat^k$. Using Proposition \ref{prop:polar_max}, we can write the above scheme as
\begin{center}
	\begin{boxedminipage}{8.75cm}
		\begin{align}\label{alg:subgrad_pd}
		\begin{split}
		{\rm (CondGradU)}~~&\projMat^{k+1} = \operatorname{PD}\bigxiaokuohao{\dataMat\subGMat^k},~{\rm where}\\
		&~~~~~~~\subGMat^k\in  \partial \bigonenorm{V}\Big |_{V = \dataMatTrans \projMat^k},
		\end{split}
		\end{align}
	\end{boxedminipage}
\end{center}
or simply write it in the following more  compact format:
\begin{align*}
\label{alg:subgradu_pd_fixed_point_version}
\projMat^{k+1} = \operatorname{PD}\bigxiaokuohao{\dataMat\partial \bigonenorm{V}\Big |_{V = \dataMatTrans \projMat^k}}.
\end{align*}

If one computes $\subgradMat^k =  \operatorname{sgn}\bigxiaokuohao{\dataMatTrans\projMat^k}$ in \eqref{alg:subgrad_pd}, then it is clear that $\dataMat\subgradMat^k\in \partial F(\projMat^k)$, and hence we conclude that:
\begin{proposition}
	NGA of \cite{nie2011robust}, which has been equivalently formulated as a fixed-point iteration: $\projMat^{k+1} = \operatorname{PD}\bigxiaokuohao{ \dataMat \operatorname{sgn}\bigxiaokuohao{\dataMat^{\top}\projMat^k  }  }$ in \eqref{alg:fixed_point_iteration_l1_pca}, is a special instance of  CondGradU. 
\end{proposition}

\subsubsection{NGA \eqref{alg:fixed_point_iteration_l1_pca} as an alternating maximization} We now show that NGA \eqref{alg:fixed_point_iteration_l1_pca} can be regarded as an alternating maximization method. To see this,   first using the fact that the $L_1$-norm is dual to  the $L_\infty$-norm, i.e., $\bigonenorm{Y} =\max_{ \bignorm{Z}_{\infty}\leq 1  }\innerprod{Y}{Z} $, one can rewrite the objective function of $L_1$-PCA \eqref{prob:l1_pca}  as 
\[
\bigonenorm{\dataMat^{\top}\projMat} = \max_{\biginfnorm{\subGMat}\leq 1}\innerprod{\dataMatTrans\projMat}{\subGMat}.
\]
Thus $L_1$-norm PCA can be equivalently formulated as the following bilinear maximization problem:
\begin{align}
\label{prob:l1_pca_bilinear}
&\max_{\projMat,\subGMat}~F(\projMat,\subGMat):=  \innerprod{\dataMat^{\top}\projMat}{\subGMat} ~{\rm s.t.}~\projMat \in\st{d,K}, \biginfnorm{\subGMat}\leq 1.
\end{align}
Applying the alternating maximization method, one may alternatively compute
\begin{align}
\label{alg:am_prototype}
\begin{split}
&  {\rm 1.~Compute}~ \subGMat^{k} \in \arg\max_{\biginfnorm{\subGMat}\leq 1}\innerprod{\dataMatTrans\projMat^k}{\subGMat}\\
& {\rm 2. ~Compute}~\projMat^{k+1} \in \arg\max_{\projMat\in\st{d,K}}\innerprod{\dataMat\subGMat^{k}}{\projMat}.
\end{split}
\end{align}
Clearly,   $\subGMat^{k} = \operatorname{sgn}\bigxiaokuohao{\dataMatTrans\projMat^k}$ and $\projMat^{k+1} = \operatorname{PD}\bigxiaokuohao{\dataMat\subGMat^k}$ respectively solve the two subproblems above. The above observations  show that:
\begin{proposition}
	NGA of \cite{nie2011robust}  is a special instance of  the alternating maximization. 
\end{proposition}   

\subsection{Finite-step convergence of NGA} \label{sec:finite_step_conv_fpi}

\subsubsection{Convergence results from the conditional (sub)gradient perspective} We first establish the subsequential convergence of CondGradU \eqref{alg:subgrad_pd}.     In the sequel, by \emph{subsequential convergence}, we mean that every limit point of the sequence generated by an algorithm satisfies certain optimality conditions. 

We shall remark that \cite{luss2013conditional} had established the subsequential convergence   for CondGradU \eqref{alg:condg-u_prototype} when the objective function is continuously differentiable; however, this result  cannot be applied, due to the non-differentiability of $\bigonenorm{\dataMatTrans\projMat}$. Nevertheless, by exploring the structure of $L_1$-norm PCA, subsequential convergence can still be obtained. 
For convenience we  may write $\dataMat = [\mathbf x_1,\ldots,\mathbf x_n]$ with $\mathbf x_i\in\mbr{d}$ and $\projMat = [\mathbf u_1,\ldots,\mathbf u_K]$ with $\mathbf u_i\in\mbr{d}$ in the sequel. 
\begin{proposition}[Subsequential convergence of CondGradU \eqref{alg:subgrad_pd}]
	\label{prop:subsequential_conv}
	Let $\{ \projMat^k \}$ be generated by CondGradU \eqref{alg:subgrad_pd}.   Then   $F(\projMat^k) $ is monotonically increasing, and  every limit point satisfies the FOC \eqref{eq:FOC}. 
\end{proposition}
\begin{proof}
	$\bigdakuohao{F(\projMat^k)}$ being monotonically increasing follows from \cite{luss2013conditional,nie2011robust}. To prove the second claim, since $\bigdakuohao{F(\projMat^k)}$ is bounded over $\st{d,K}$, we have that $F(\projMat^{k+1}) - F(\projMat^k)\rightarrow 0$. On the other hand, the definition of $\projMat^{k+1}$ shows that $\innerprod{\dataMat\subGMat^k}{\projMat^{k+1}}\geq \innerprod{\dataMat\subGMat^k}{\projMat},\forall\projMat\in\st{d,K}$, and so
	\begin{align}
	\label{eq:proof_subsequential_conv_1}
	\begin{split}
	&\innerprod{\dataMat\subGMat^k}{\projMat - \projMat^k}\leq \innerprod{\dataMat\subGMat^k}{\projMat^{k+1} - \projMat^k} \\
	\leq &F(\projMat^{k+1}) - F(\projMat^k) \rightarrow 0,~\forall\projMat\in\st{d,K},
	\end{split}
	\end{align}
	where the second inequality follows from the convexity of $F(\projMat)$, the definition of subdifferential \eqref{eq:subdifferential},  and $\dataMat\subGMat^k\in\partial F(\projMat^k)$. Since $\{ \projMat^k\}\in \st{d,K}$ is bounded, limit points exist.   Let $\projMat^*$ be a limit point of $\{\projMat^k \}$ and assume that $\{ \projMat^{k_l}  \}\rightarrow \projMat^*$ as $l\rightarrow\infty$.  We then consider $\{\subGMat^{k_l}  \}$, which is also bounded.  Passing to a subsequence of $\{ \subGMat^{k_l} \}$ if necessary, we can without loss of generality also assume that $\{\subGMat^{k_l} \}$ itself   converges and denote the limit   as $\subGMat^*$.  Thus,  in \eqref{eq:proof_subsequential_conv_1},  letting $k=k_l$ and letting $l\rightarrow\infty$, we obtain 
	\begin{align}
	\label{eq:proof_subsequential_conv_2}
	\innerprod{\dataMat\subGMat^*}{\projMat - \projMat^*} \leq 0,~\forall \projMat\in\st{d,K}.
	\end{align}
	It remains to show that $\subGMat^* \in \partial \bigonenorm{V}\Big |_{V=\dataMatTrans\projMat^*}$, i.e., $\subGMat^*_{ij}\in \partial\bigjueduizhi{ V_{ij}} \Big |_{V_{ij} = (\dataMatTrans\projMat^*)_{ij} = \mathbf x_i^{\top}\mathbf u^*_j}$ for each $i,j$. 
	
	We devide the proof into three cases. If $\mathbf x_i^{\top}\mathbf u_j^*>0$, then since $\mathbf u^{k_l}_j\rightarrow \mathbf u^*_j$, there exists a large enough interger $l_0$ such that whenever $l>l_0$, $\mathbf x_i^{\top}\mathbf u^{k_l}_j>0$, and so $\subGMat^{k_l}_{ij} = \partial \bigjueduizhi{ V_{ij} } \Big |_{V_{ij}=\mathbf x_i^{\top}\mathbf u^k_j  } = 1$; as a result, we conclude that $\subGMat^*_{ij} = \lim_{l\rightarrow\infty}\subGMat^{k_l}=1$, and hence $\subGMat^*_{ij}\in \partial\bigjueduizhi{ V_{ij}} \Big |_{V_{ij} =   \mathbf x_i^{\top}\mathbf u^*_j}$. 
	
	If $\mathbf x_i^{\top}\mathbf u^*_j <0$, using similar argument we get $\subGMat^*_{ij} = -1\in \partial\bigjueduizhi{ V_{ij}} \Big |_{V_{ij} =   \mathbf x_i^{\top}\mathbf u^*_j}$. If $\mathbf x^{\top}_i\mathbf u^*_j=0$, then $\partial \bigjueduizhi{V_{ij}} \Big |_{V_{ij}=\mathbf x^{\top}_i\mathbf u^*_j} = [-1,1]$, while we always have $\subGMat^{k_l}_{ij} \in [-1,1]$, and so $\subGMat^{k_l}_{ij}\rightarrow \subGMat^{*}_{ij} \in \partial \bigjueduizhi{V_{ij}} \Big |_{V_{ij}=\mathbf x^\top_i\mathbf u^*_j}$. 
	
	As a result, we obtain	$\subGMat^* \in \partial \bigonenorm{V}\Big |_{V=\dataMatTrans\projMat^*}$, which together with \eqref{eq:proof_subsequential_conv_2} is exactly the FOC \eqref{eq:FOC}. 
\end{proof}

We then establish the finite-step convergence. This is only specialized to   NGA $\projMat^{k+1} = \operatorname{PD}\bigxiaokuohao{ \dataMat \operatorname{sgn}\bigxiaokuohao{\dataMat^{\top}\projMat^k  }  }$. 
Our analysis is based on a simple observation: If $\subGMat^k = \operatorname{sgn}\bigxiaokuohao{\dataMatTrans\projMat^k}$, then there exists at most $3^{nK}$ possible choices of $\subGMat^k$, which is finite, and so the number of possible matrices $\dataMat\subGMat^k$ is also finite. This   would then give finitely many possible $\projMat^k$ if it can be uniquely determined by $\subGMat^k$. We have the following results.
\begin{theorem}[Finite-step convergence of NGA \eqref{alg:fixed_point_iteration_l1_pca}]
	\label{th:global_conv}
	Let $\{ \projMat^k \}$ be generated by NGA \eqref{alg:fixed_point_iteration_l1_pca}: $\projMat^{k+1} = \operatorname{PD}\bigxiaokuohao{ \dataMat \operatorname{sgn}\bigxiaokuohao{\dataMat^{\top}\projMat^k  }  }$ (interpreted as   CondGradU \eqref{alg:subgrad_pd}). If $\rank{\dataMat\subGMat^k} =K$ for all $k$, then after finitely many steps,   the algorithm finds a point $\projMat^*$ which is a FOC point of the form \eqref{eq:FOC}.
\end{theorem}
\begin{proof}
	For any $k$, write $\dataMat\subGMat^k = \projMat^{k+1}H^{k+1}$ according to PD, where $H^k\in\PSD{K\times K}$. Then Theorem \ref{th:polar_dec} tells us that $H^k$ is  uniquely determined by $\dataMat\subGMat^k$.	Based on the discussions above this theorem, the number of possible $\dataMat\subGMat^k$'s is finite, and so the number of possible $H^k$'s is also finite. Thus we can define $\lambda:= \min_{k} \lambda_{\min}(H^k)$. Since $\rank{\dataMat\subGMat^k} =K$, all the $H^k\in\PSDD{K\times K}$ and so $\lambda>0$. Using Lemma \ref{lem:polar_max_sufficient} and the definition of $\lambda$, we have
	\begin{align}
	\label{eq:proof_global_conv_1}
	\begin{split}
	\frac{\lambda}{2}\bigfnormsquare{\projMat^{k+1}-\projMat^k} &\leq \innerprod{\dataMat\subGMat^k}{\projMat^{k+1}-\projMat^k} \\ 
	&\leq  F(\projMat^{k+1}) - F(\projMat^k),
	\end{split}
	\end{align}
	which together with $F(\projMat^{k+1})-F(\projMat^k)\rightarrow 0 $ gives that $\bigfnorm{\projMat^{k+1}-\projMat^k}\rightarrow 0$. Since $\{\projMat^k  \}$ is bounded, it follows from \cite[Theorem 8.3.9]{facchinei2007finite} that the set of limit points of $\{\projMat^k \}$ is closed and connected. 
	
	On the other hand, Proposition \ref{prop:polar_max} (or Lemma \ref{lem:polar_max_sufficient}) shows that every $\dataMat\subGMat^k$ gives a unique $\projMat^{k+1}$ provided that $\rank{\dataMat\subGMat^k}=K$. Thus the number of possible $\projMat^k$'s is also finite, which implies that the set of limit points of $\{\projMat^k \}$ is   finite and discrete. Taking the above discussions together, we conclude that there exists only one limit point for $\{\projMat^k  \}$, termed as $\projMat^*$, with $\projMat^k\rightarrow\projMat^*$. $\projMat^*$ being a FOC \eqref{eq:FOC} follows from Proposition \ref{prop:subsequential_conv}. In fact, more can be obtained:   using the facts that $F(\projMat^{k+1})-F(\projMat^k)\rightarrow 0 $ and $\{\projMat^k  \}$ being finite again, we can also conclude from \eqref{eq:proof_global_conv_1} that there is a $k_0$, such that when $k\geq k_0$,  $\bigfnorm{\projMat^k - \projMat^{k+1}} =0$. Thus $\projMat^{k_0} = \projMat^{k_0+1}=\cdots = \projMat^*$, i.e., after finitely many steps, the algorithm finds a FOC point.   
\end{proof}
\begin{remark}
	We discuss the reality of the assumption $\rank{\dataMat\subGMat^k}=K$  in Theorem \ref{th:global_conv}. Recall that $\dataMat\in\mbr{d\times n}$ and $\subGMat^k\in\mbr{n\times K}$. If $d\geq n\geq K$, and if $\rank{\dataMat}=n$ and $\rank{\subGMat^k}=K$, then $\rank{\dataMat\subGMat^k}=K$. Since $n$ is the number of samples, $d$ is the dimension of each sample, and $K$ is the projected dimension, in practice usually $d\geq n\geq K$. Next, it is known that   a generic matrix $A\in\mbr{p\times q}$ satisfies $\rank{A}=\min\{p,q \}$. Thus $\rank{\dataMat}=n$ generically. However, if the data $\{ \mathbf x_i \}_{i=1}^n$ are centralized, then $\rank{\dataMat}\leq K-1$. Nevertheless, we have observed from extensive simulations that even if $\rank{\dataMat}\leq K-1$,   there still hold $\rank{\dataMat\subGMat^k}=K$ (and $\rank{\subGMat^k}=K$)  for each $k$.  Thus one should not consider $\rank{\dataMat\subGMat^k}=K$ as a stringent assumption.  In particular, we show in the following that this assumption can be removed when $K=1$. 
\end{remark}
{\bf $K=1$ case}  Now,	consider  the special case that $K=1$ in $L_1$-norm PCA \eqref{prob:l1_pca}, namely, the model \eqref{prob:l1_pca_K1}, where the algorithm is formulated in \eqref{alg:fixed_point_l1_pca_K_1}. 
In this particular case the finite-step convergence holds without any assumption. 
\begin{theorem}
	\label{th:global_conv_K_1}
	Let $K=1$ and let $\bigdakuohao{ \mathbf u^k }$ be generated by the scheme \eqref{alg:fixed_point_l1_pca_K_1}. Choose   an initializer $\mathbf u^0$     such that $\dataMatTrans\mathbf u^0\neq 0$. Then the finite-step convergence results in Theorem \ref{th:global_conv} apply. 
\end{theorem}  
\begin{proof}
	Let $c:= \bigonenorm{\dataMatTrans\mathbf u^0}>0$.	We first show that $\bigfnorm{\dataMat\mathbf s^k}\geq c $ for all $k$. By the definitions of $\mathbf u^{k+1}$ and $\mathbf s^k$ in \eqref{alg:fixed_point_l1_pca_K_1}, we have
	\begin{align*}
	\bigfnorm{\dataMat\mathbf s^k} &= \innerprod{\dataMat\mathbf s^k}{\frac{\dataMat\mathbf s^k }{\bigfnorm{\dataMat\mathbf s^k}   }} = \innerprod{\dataMat\mathbf s^k}{\mathbf u^{k+1}} \\ &\geq \innerprod{\dataMat\mathbf s^k}{\mathbf u^k} = \innerprod{\mathbf s^k}{\dataMatTrans\mathbf u^k} \\& = \bigonenorm{\dataMatTrans\mathbf u^k}, 
	\end{align*} 
	where the first inequality follows from that $\mathbf u^{k+1}$ maximizes $\max_{\|\mathbf u\|=1}\innerprod{\dataMat\mathbf s}{\mathbf u}$. 
	Since \eqref{alg:fixed_point_l1_pca_K_1} is the $K=1$ case of the scheme \eqref{alg:fixed_point_iteration_l1_pca}, Proposition \ref{prop:subsequential_conv} shows that $\{ \bigonenorm{\dataMatTrans\mathbf u^k} \}$ is monotonically increasing, and so we get that $\bigfnorm{\dataMat\mathbf s^k}\geq \bigonenorm{\dataMatTrans\mathbf u^k} \geq \bigonenorm{\dataMatTrans\mathbf u^0} = c>0$. Next, using Corollary \ref{col:polar_max_sufficient_n1}, we   have
	\begin{align*}
	F(\mathbf u^{k+1}) - F(\mathbf u^k)  &= \innerprod{\dataMat\mathbf s^k}{\mathbf u^{k+1}  }- \innerprod{\dataMat\mathbf s^k}{  \mathbf u^k}  \\
	&\geq \frac{c}{2}\bigfnorm{\mathbf u^{k+1}-\mathbf u^k}^2,~\forall k,
	\end{align*}
	The remaining argument is similar to Theorem \ref{th:global_conv}. 
\end{proof}

We  then turn to   the perspective of alternating maximization, based on  which we can remove the full rank assumption and obtain an   upper bound on the number of steps.

\subsubsection{Convergence results from the alternating maximization perspective} 
Unlike the previous part, we do not make any assumption here. 
We have the following results that give an upper bound on the number of steps, leading to a stronger result with less assumptions than Theorem \ref{th:global_conv}.  First recall the formulation \eqref{prob:l1_pca_bilinear} and that  $\subGMat^{k} = \operatorname{sgn}\bigxiaokuohao{\dataMatTrans\projMat^k}$ and $\projMat^{k+1} = \operatorname{PD}\bigxiaokuohao{\dataMat\subGMat^k}$ respectively solve the two subproblems in \eqref{alg:am_prototype}.  The idea in this part is simple: the gap between $F(\projMat^{k+1},\subGMat^{k+1})$ and $F(\projMat^{k+1},\projMat^k)$ can be explicitly bounded and the possible $\subGMat^k$ (and even $\projMat^k$) are finite; then the number of steps can be estimated by these two observations. Before stating the results, we make the following setting in the sequel and discuss it in Remark \ref{rmk:setting}. 
\begin{setting}\label{set:1}
	We make the setting that in polar decomposition,  we always select the same $\projMat$-factor for the same input matrix.  
\end{setting}

Recalling that $F^{\max}$ represents the global maximum of $L_1$-norm PCA, we have:
\begin{theorem}[Finite-step convergence of NGA \eqref{alg:fixed_point_iteration_l1_pca}]
	\label{th:finite_conv_original_alg_fixed_point_iteration}
	Let $\{ \projMat^k, \subGMat^k \}$ be generated by  NGA \eqref{alg:fixed_point_iteration_l1_pca} (interpreted as an alternating maximization \eqref{alg:am_prototype}
	), where the initializer $(\projMat^0,\subGMat^0)$ is chosen such that $F(\projMat^0,\subGMat^0)> 0$. With Setting \ref{set:1}, 
	the following statements hold: 
	\begin{enumerate}
		\item There exists a positive integer $k_0$, such that for all $i,j$ and $\forall k\geq k_0$, whenever $\mathbf x^\top_i\mathbf u_j^{k+1}\neq 0$, it holds that $\subGMat^{k+1}_{ij}=\subGMat^{k}_{ij}$;
		\item  There exists a constant $\tau_0>0$, such that after at most $\left\lceil   \frac{F^{\max}}{\tau_0} \right\rceil$ steps, there holds $F(\projMat^k,\subGMat^k) = F(\projMat^{k+1},\subGMat^{k+1})=F(\projMat^{k+2},\subGMat^{k+2})=\cdots$,  while before it stops, the objective value increases at least the quantity of $\tau_0$ in each iteration;  
		\item Let $k^\prime$ denote the first step that the objective value does not increase anymore. Then $\projMat^{k^\prime +1}$ is a FOC point of the form \eqref{eq:FOC_fpf}. 
	\end{enumerate}
\end{theorem}
\begin{remark} \label{rmk:setting}We first add a comment on Setting \ref{set:1}. In fact, the validness of   Theorem \ref{th:finite_conv_original_alg_fixed_point_iteration} relies on the uniqueness of the $\projMat$-factor when computing PD of a given matrix. 
	However, Remark \ref{rmk:pd_error_estimation} states that if a matrix $C$ is not of full column rank, then there exist infinitely many $U$-factors in PD of $C$. Nevertheless, this will not be a trouble, as one can always chooses a fixed one in each time one computes   PD of a given matrix $C$, e.g., one only computes PD of $C$ for the first time that $C$ appears, and stores the $U$-factor;   the next time that $C$ appears again, one can directly fetch the stored $U$-factor.  Another more convenient method is to use any deterministic algorithm to comput PD that always returns the same $\projMat$-factor for the same input matrix.  
	
	Moreover, there is even no need to emphasize the setting: 
	As will be revealed in Proposition \ref{prop:no_repeat_U}, no repeated $\projMat^k$ will occur before the algorithm stops. Thus the uniqueness of the $\projMat$-factor makes sense. 
\end{remark}

For proving the theorem, we denote $\boldsymbol{U}:=\{ \projMat^k \}$ the sequence of projection matrices generated by the algorithm. Since the number of possible $\subGMat^k$ is finite, with Setting \ref{set:1}, $\boldsymbol{U}$ is a discrete set. In the sequel, we denote
\begin{align}\label{eq:tau0}
\tau_0:= \min\bigdakuohao{ \bigjueduizhi{ \mathbf x^{\top}_i\mathbf u_j     }  ~\mid ~\mathbf x^{\top}_i\mathbf u_j\neq 0, \projMat\in \boldsymbol{\projMat}     }.
\end{align}
Clearly, $\tau_0$ is well-defined and $\tau_0>0$.

\begin{proof}[Proof of Theorem \ref{th:finite_conv_original_alg_fixed_point_iteration}]
	Item 1: Recall that the algorithm is written as $\subGMat^k = \operatorname{sgn}\bigxiaokuohao{\dataMatTrans\projMat^k}$ and $\projMat^{k+1} = \operatorname{PD}\bigxiaokuohao{\dataMat\subGMat^k}$. 
	We observe for all $k$ that
	\begin{align}
	&F(\projMat^{k+1},\subGMat^{k+1}) - F(\projMat^{k+1},\subGMat^k) \\=& \innerprod{\dataMatTrans\projMat^{k+1}}{\subGMat^{k+1}} - \innerprod{\dataMatTrans\projMat^{k+1}}{\subGMat^k}\nonumber\\
	=& \sum_{ \mathbf x^{\top}_i\mathbf u^{k+1}_j\neq 0 } \bigxiaokuohao{\bigjueduizhi{  \mathbf x^{\top}_i\mathbf u^{k+1}_j} - \mathbf x^{\top}_i\mathbf u^{k+1}_j \subGMat^k_{ij}}. \nonumber
	\end{align}
	Now, for $(i,j)$ such that $\mathbf x^{\top}_{i}\mathbf u^{k+1}_j\neq 0$,  
	$\subGMat^k_{ij}$ has three possible values: $\subGMat^k_{ij}=\subGMat^{k+1}_{ij}$, $\subGMat^k_{ij}=-\subGMat^{k+1}_{ij}$, and $\subGMat^k_{ij}=0$. For the first one, $  \mathbf x^{\top}_i\mathbf u^{k+1}_j\subGMat^k_{ij} =| \mathbf x^{\top}_i\mathbf u^{k+1}_j  | $; for the second one, $   \mathbf x^{\top}_i\mathbf u^{k+1}_j\subGMat^k_{ij} = - |\mathbf x^{\top}_i\mathbf u^{k+1}_j |$; for the last one, $\mathbf x^{\top}_i\mathbf u^{k+1}_j \subGMat^k_{ij} =0,$. Thus, 
	\begin{align}
	&F(\projMat^{k+1},\subGMat^{k+1}) - F(\projMat^{k+1},\subGMat^k) \nonumber\\
	=& \sum_{ \mathbf x^{\top}_i\mathbf u^{k+1}_j\neq 0 } \bigxiaokuohao{\bigjueduizhi{  \mathbf x^{\top}_i\mathbf u^{k+1}_j} - \mathbf x^{\top}_i\mathbf u^{k+1}_j \subGMat^k_{ij}} \nonumber\\
	=& \sum_{ \mathbf x^{\top}_i\mathbf u^{k+1}_j\neq 0, \subGMat^{k+1}_{ij}+\subGMat^k_{ij}=0 } 2\bigjueduizhi{ \mathbf x^{\top}_i\mathbf u^{k+1}_j   }+ \!\!\!\!\!\!  \sum_{ \mathbf x^{\top}_i\mathbf u^{k+1}_j\neq 0, \subGMat^k_{ij}=0 }  \bigjueduizhi{ \mathbf x^{\top}_i\mathbf u^{k+1}_j   }. \label{eq:proof_finite_conv_origin_1}
	\end{align}
	Since   $F(\projMat^{k+1},\subGMat^{k+1}) - F(\projMat^{k+1},\subGMat^k)\rightarrow 0$   and    $\tau_0 $ in \eqref{eq:tau0} is a strictly positive constant, \eqref{eq:proof_finite_conv_origin_1} implies that there exists a $k_0$, such that when $k\geq k_0$, both sides of \eqref{eq:proof_finite_conv_origin_1} are exactly zero (othewise, the right hand-side is always larger than $\tau_0$, which contradicts that $F(\projMat^{k+1},\subGMat^{k+1}) - F(\projMat^{k+1},\subGMat^k)\rightarrow 0$). Therefore,  when $k\geq k_0$, the only possible value of $\subGMat^k_{ij}$ is that $\subGMat^k_{ij}=\subGMat^{k+1}_{ij}$ (when $\mathbf x^\top_i\mathbf u^{k+1}_j\neq 0$), i.e., after finitely many steps, the claim of  item 1 is true.

	Item 2: we first have that for each $k$, there must exist at least a pair ${i,j}$, such that $\mathbf x^\top_i\mathbf u^k_j \neq 0$. This is clear, because $\sum_{i,j} | \mathbf x^\top_i\mathbf u^{k}_j  | =  F(\projMat^k,\subGMat^k) \geq F(\projMat^0,\subGMat^0) >0$. Therefore, by \eqref{eq:proof_finite_conv_origin_1}, whenever $F(\projMat^{k+1},\subGMat^{k+1})$ strictly increases from $F(\projMat^{k+1},\subGMat^k)$, it increases at least 
	the quantity $|\mathbf x^\top_i\mathbf u^{k+1}_j| \geq \tau_0>0$, and so when $F(\projMat^{k+1},\subGMat^{k+1})\neq F(\projMat^k,\subGMat^k)$,
	\begin{align*}
	&	F(\projMat^{k+1},\subGMat^{k+1}) - F(\projMat^k,\subGMat^k)\\ \geq& F(\projMat^{k+1},\subGMat^{k+1}) - F(\projMat^{k+1},\subGMat^k) 	\geq  \tau_0.
	\end{align*}
	Summing it from $k=0$ to   any  $K\geq 1$, we have $K\tau_0 \leq F(\projMat^{K+1},\subGMat^{K+1}) - F(\projMat^0,\subGMat^0) \leq F^{\max}$. This inequality is valid if $K\leq \left\lfloor   \frac{F^{\max}}{\tau_0} \right\rfloor$. 
	Thus      after  at most $\left\lceil   \frac{F^{\max}}{\tau_0} \right\rceil$ steps, the objective value cannot increase anymore. 
	
	Item 3:	let $k^\prime$ denote the first step that the objective value does not increase anymore, i.e., $F(\projMat^{k^\prime},\subGMat^{k^\prime})=F(\projMat^{k^\prime+1},\subGMat^{k^\prime+1})$. We show that $(\projMat^{k^\prime+1},\subGMat^{k^\prime})$ is a FOC point. Note that   \eqref{eq:proof_finite_conv_origin_1} implies that $\subGMat^{k^\prime}_{ij} = \subGMat^{k^\prime+1}_{ij}$ if $\mathbf x^\top_{i}\mathbf u^{k^\prime}_j\neq 0$, i.e., 
	$
	\subGMat^{k^\prime}_{ij} = \subGMat^{k^\prime+1} \in \partial | V_{ij} | \Big |_{V_{ij} = \mathbf x^\top_i\mathbf u^{k^\prime+1}_j }$ when  $\mathbf x^\top_i\mathbf u^{k^\prime+1}_j\neq 0. 
	$
	On the other hand, since $\subGMat^{k^\prime}_{ij} \in [-1,1]$ in any case, when $\mathbf x^\top_i\mathbf u^{k^\prime+1}_j=0$, we must have that $\subGMat^{k^\prime}_{ij} \in \partial |V_{ij}| \Big |_{V_{ij} = \mathbf x^\top_i\mathbf u^{k^\prime+1}_j }$. As a result,
	\[
	\subGMat^{k^\prime} \in \partial \bigonenorm{V } \Big |_{V=\dataMatTrans\projMat^{k^\prime + 1}}. 
	\]
	This together with $\projMat^{k^\prime+1} = \operatorname{PD}\xiaokuohao{\dataMat\subGMat^{k^\prime}}$   demonstrates that  $\projMat^{k^\prime+1}$ satisfies the FOC point of the form \eqref{eq:FOC_fpf}.   
\end{proof}
\begin{remark}
	On can also derive an upper bound on the number of steps that only depends on the data matrix $\dataMat$. Denote $$\boldsymbol{U}^*:= \left\{ \projMat = \operatorname{PD}(\dataMat\subGMat) \mid \subGMat{\rm~is~such~that~}\subGMat_{ij}\in \{0,\pm 1 \}  \right\}.$$
	Under Setting \ref{set:1}, $\boldsymbol{U}^*$ is also a discret set. Then, similar to $\tau_0$ in \eqref{eq:tau0}, denote $\tau_*:= \min\bigdakuohao{ \bigjueduizhi{ \mathbf x^{\top}_i\mathbf u_j     }  ~\mid ~\mathbf x^{\top}_i\mathbf u_j\neq 0, \projMat\in \boldsymbol{\projMat}^*     }$. It is clear that $\tau_0\geq \tau_*$, and so  $\left\lceil   \frac{F^{\max}}{\tau_0} \right\rceil \leq   \left\lceil   \frac{F^{\max}}{\tau_*} \right\rceil $, which is essentially only dependent on $\dataMat$.  
\end{remark}

A direct conclusion of Theorem \ref{th:finite_conv_original_alg_fixed_point_iteration} is:
\begin{corollary}\label{col:f_donot_change_foc_point}
	Whenever $F(\projMat^{k+1},\subGMat^{k+1}) - F(\projMat^{k+1},\subGMat^k) = 0$,   $\projMat^{k+1}$ is a FOC point of the form \eqref{eq:FOC_fpf}.
\end{corollary}
Theorem \ref{th:finite_conv_original_alg_fixed_point_iteration} means that $F(\projMat^{k+1},\subGMat^{k+1}) - F(\projMat^{k+1},\subGMat^k) = 0$ can be used as a stopping criterion for NGA \eqref{alg:fixed_point_iteration_l1_pca}. However, $\bigfnorm{\subGMat^{k+1}-\subGMat^k}=0$ may not be used as the stopping criterion, as   Theorem \ref{th:finite_conv_original_alg_fixed_point_iteration} does not make sure this. The reason can be seen from the proof of Theorem \ref{th:finite_conv_original_alg_fixed_point_iteration}: if $\mathbf x^{\top}_i\mathbf u^{k+1}_j=0$, then $\subGMat^k_{ij}$ may not be equal to $\subGMat^{k+1}_{ij}$.   

Next, we remark that, there is even no need to emphasize that the algorithm works under Setting \ref{set:1}, due to the following reason:
\begin{proposition}\label{prop:no_repeat_U}
	Let $\{ \projMat^k,\subGMat^k\}$ be generated by NGA \eqref{alg:fixed_point_iteration_l1_pca} (interpreted as \eqref{alg:am_prototype}) with $F(\projMat^0,\subGMat^0)>0$.  If during the process, there are two indices $k^\prime< k^{\prime\prime}$, such that $\projMat^{k^\prime} = \projMat^{k^{\prime\prime}}$, then $\projMat^{k^\prime+1}$ is a FOC point of the form \eqref{eq:FOC_fpf}.
\end{proposition}
\begin{proof}
	Since $\subGMat^k = {\rm sgn}(\dataMatTrans\projMat^k)$, this together with $\projMat^{k^\prime} = \projMat^{k^{\prime\prime}}$ gives that $F(\projMat^{k^\prime},\subGMat^{k^\prime})=F(\projMat^{k^{\prime\prime}},\subGMat^{k^{\prime\prime}})$, which further implies that $F(\projMat^{k^\prime+1},\subGMat^{k^\prime+1}) = F(\projMat^{k^\prime+1},\subGMat^{k^\prime}) = F(\projMat^{k^\prime},\subGMat^{k^\prime})$. Similar to the proof of Theorem \ref{th:finite_conv_original_alg_fixed_point_iteration}, we can conclude that $\projMat^{k^\prime+1}$ is a FOC point of the form \eqref{eq:FOC_fpf}. 
\end{proof}
As a result of Proposition \ref{prop:no_repeat_U}, before we encounter a FOC point, there does not exist repeated $\projMat^k$ generated by the algorithm.  

Finally, we   remark that finite-step convergence     implies linear convergence to a certain extent. Assume that $k^\prime$ is the   index such that $F(\projMat^{k^\prime+1},\subGMat^{k^\prime+1}) - F(\projMat^{k^\prime+1},\subGMat^{k^\prime})=0$, i.e., $\projMat^{k^\prime+1}$ is a FOC point.   Denote 
\begin{align*}
\rho := \max_{0\leq k\leq k^\prime-1}  \frac{ F(\projMat^{k^\prime+1},\subGMat^{k^\prime+1}) - F(\projMat^{k+1},\subGMat^{k+1})  }{F(\projMat^{k^\prime+1},\subGMat^{k^\prime+1}) - F(\projMat^{k},\subGMat^{k})   }.
\end{align*}
Thus $0<\rho<1$ because  the objective value increases strictly before it stops by Proposition \ref{prop:subsequential_conv}. Therefore,
\begin{align*}
&	F(\projMat^{k^\prime+1},\subGMat^{k^\prime+1}) - F(\projMat^{k+1},\subGMat^{k+1})\\  
\leq& \rho \bigxiaokuohao{ F(\projMat^{k^\prime+1},\subGMat^{k^\prime+1}) - F(\projMat^{k},\subGMat^{k})   } \\
\leq& \cdots \leq \rho^{k+1} \bigxiaokuohao{ F(\projMat^{k^\prime+1},\subGMat^{k^\prime+1}) - F(\projMat^{0},\subGMat^{0})   }. 
\end{align*}
\begin{corollary}
	With respect to the objective value,  NGA \eqref{alg:fixed_point_iteration_l1_pca} (interpreted as \eqref{alg:am_prototype}) converges linearly before it stops. 
\end{corollary}

Comparing the main convergence results obtained in this section, i.e., Theorems \ref{th:global_conv} and \ref{th:finite_conv_original_alg_fixed_point_iteration}, we obtain from Theorem \ref{th:global_conv} that   when interpreted as a CondGradU,  the iterative points are proved to stop in finitely many steps, if a full rank assumption is satisfied; when viewed as an alternating maximization,  Theorem \ref{th:finite_conv_original_alg_fixed_point_iteration} ensures that the objective value will not change in finitely many steps without assumptions, but the iterative points might not, as discussed below Corollary  \ref{col:f_donot_change_foc_point}.  To achieve  finite-step convergence on the iterative points while avoiding     assumptions,  we consider a   slightly modification for NGA \eqref{alg:fixed_point_iteration_l1_pca} in the next subsection.  

\subsection{Finite-step convergence of  $\subGMat$-PNGA}
Given the interpretation of NGA \eqref{alg:fixed_point_iteration_l1_pca} as an alternating maximization \eqref{alg:am_prototype}, we   consider a simple modification:   substracting a proximal term from the $\subGMat$-subproblem.  To emphasize the importance of $\subGMat$ in the coming analysis, we switch the order of the two subproblems; we also slightly modify the constraint of the $\subGMat$-subproblem, resulting in:
\begin{align}
\label{alg:am_prox_prototype}
\begin{split}
& {\rm 1. ~Compute}~\projMat^{k+1} \in \arg\max_{\projMat\in\st{d,K}}\innerprod{\dataMat\subGMat^{k}}{\projMat}  \\
&  {\rm 2.~Compute}~ \subGMat^{k+1} \in \arg\max_{ \{ \subGMat\mid \subGMat_{ij}\in \{0,\pm 1 \} \} }\innerprod{\dataMatTrans\projMat^{k+1}}{\subGMat} \\
&~~~~~~~~~~~~~~~~~~~~~~~~~~~~~- \frac{\tau}{2}\bigfnormsquare{\subGMat - \subGMat^{k}},
\end{split}
\end{align} 
where $\tau>0$ is a given parameter.  After   calculation, the $\subGMat$-subproblem can be equivalently rewritten as $\min_{ \{ \subGMat\mid \subGMat_{ij}\in \{0,\pm 1 \} \}}\bigfnorm{ \subGMat - \bigxiaokuohao{ \tau\subGMat^k + \dataMatTrans\projMat^{k+1}   }   }$, and clearly, $\subGMat^{k+1 } = \operatorname{sgn}\bigxiaokuohao{ \tau\subGMat^{k} + \dataMatTrans\projMat^{k+1}  }$  solves this subproblem. In this regard, the   algorithm is given as follows:

\begin{center}
	\begin{boxedminipage}{8.75cm}
		\begin{align}
		\label{alg:am_prox_l1_pca}
		\begin{split}
		{\rm (}\subGMat{\rm -PNGA)}~~	&\projMat^{k+1}=\operatorname{PD}\bigxiaokuohao{\dataMat\subGMat^k},\\
		&\subGMat^{k+1} = \operatorname{sgn}\bigxiaokuohao{\tau\subGMat^k+ \dataMatTrans\projMat^{k+1}   }.
		\end{split}
		\end{align}
	\end{boxedminipage}
\end{center}

When $\tau=0$, the above scheme reduces to NGA \eqref{alg:fixed_point_iteration_l1_pca}. The scheme \eqref{alg:am_prox_l1_pca} will be termed as $\subGMat$-PNGA in the following, where the prefix ``$\subGMat$-P'' means that the proximal term is only imposed on the $\subGMat$-subproblem.

In general, imposing   proximal terms will make it easier to establish   subsequential and global convergence of   optimization algorithms. In the context of $L_1$-norm PCA, the idea of imposing   proximal terms is motivated by \cite{wang2021linear}; however, different from \cite{wang2021linear}, here   the proximal term is only imposed on the $\subGMat$-subproblem. The reason is driven by theory, which will be   seen in Theorem \ref{th:finite_step_conv} and Remark \ref{rmk:no_prox_u_subproblem}. 
We begin with     the following simple observation: 
\begin{proposition}
	\label{prop:sk_sk1_geq2}
	Let $\subGMat^\prime,\subGMat^{\prime\prime}$ be two matrices of the same size with $\subGMat^\prime_{ij}\in \{-1,0,1\}$ and $\subGMat^{\prime\prime}_{ij}\in \{-1,0,1\}$ for each $i,j$. Then either $\subGMat^\prime = \subGMat^{\prime\prime}$, or $\bigfnorm{\subGMat^\prime - \subGMat^{\prime\prime}}\geq 1$. 
\end{proposition}

The above observation in fact implies the mechanism of the proximal term used in \eqref{alg:am_prox_prototype}: it will together with the increasing property of the objective value force the $\subGMat$ variable (and so $\projMat$) fixed after finitely many steps.   Detailed results of the finite-step convergence of $\subGMat$-PNGA \eqref{alg:am_prox_l1_pca} are given in the following theorem.
\begin{theorem}[Finite-step convergence of $\subGMat$-PNGA \eqref{alg:am_prox_l1_pca}]
	\label{th:finite_step_conv}
	Let $\bigdakuohao{\projMat^k,\subGMat^k}$ be generated by $\subGMat$-PNGA \eqref{alg:am_prox_l1_pca} with $\tau>0$, where the initializer $(\projMat^0,\subGMat^0)$ is given such that $F(\projMat^0,\subGMat^0)\geq 0$. Then 
	\begin{enumerate}
		\item There must exist a $k_0>0$, such that  $\subGMat^{k_0} = \subGMat^{k_0+1} = \subGMat^{k_0+2} = \cdots$;
		\item With Setting \ref{set:1},  after at most $\left\lceil \frac{2F^{\max}}{\tau}\right\rceil$ steps, there must hold $\projMat^k=\projMat^{k+1}=\projMat^{k+2}=\cdots$ and $\subGMat^k=\subGMat^{k+1}=\subGMat^{k+2}=\cdots$; before the algorithm stops, the objective value decreases at least the quantity $\frac{\tau}{2}$ in each iteration;
		\item With Setting \ref{set:1}, there   exists  a constant $\tau_1>0$, such that whenever $0<\tau<\tau_1$, after at most  $\left\lceil \frac{2 F^{\max}}{\tau}\right\rceil$ steps, the algorithm stops at a FOC point of the form \eqref{eq:FOC_fpf}. More precisely, let $k^\prime$ be the first step that $\projMat^{k^\prime} = \projMat^{k^\prime + 1}$. Then $\projMat^{k^\prime+1}$ is a FOC point of the form \eqref{eq:FOC_fpf}, and $(\projMat^{k^\prime + 1},\subGMat^{k^\prime+1})$ is 
		also a partial maximizer of \eqref{prob:l1_pca_bilinear}. 
	\end{enumerate}
\end{theorem}
\begin{remark}

	In $\subGMat$-PNGA \eqref{alg:am_prox_l1_pca}, even if $\subGMat^k$ will not change when $k\geq k_0$, $\projMat^k$ might not stop.  For example, if letting $\subGMat^0 = \operatorname{sgn}(\dataMatTrans\projMat^0)$ and $\tau > \max_{i,j}|\mathbf x^{\top}_i\mathbf u_j |$, $\forall \projMat \in \st{d,K}$, then we always have $\subGMat^0=\subGMat^1=\cdots$, while if $\rank{\dataMat\subGMat^0}<K$, then any $\projMat$-factor of PD of $  \dataMat\subGMat^0 $ can be used as $\projMat^k$. Thus we work with Setting \ref{set:1}. 
	
	Next, even if both $\projMat^k$ and $\subGMat^k$ stop  in finitely many steps, the output may not be a FOC point. The reason is still due to a large $\tau$. Therefore, to make the algorithm finds   a FOC point in finitely many steps, $\tau$ cannot be too large. 
\end{remark}
\begin{proof}[Proof of Theore \ref{th:finite_step_conv}]
	Item 1: By the definition of $\subGMat^{k+1}$ in $\subGMat$-PNGA \eqref{alg:am_prox_l1_pca}, we have
	\begin{align*}
	&\innerprod{\dataMatTrans\projMat^{k+1}}{\subGMat^{k+1}} - \frac{\tau}{2}\bigfnormsquare{\subGMat^{k+1}-\subgradMat^k} \\ 
	\geq& \innerprod{\dataMatTrans\projMat^{k+1}}{\subGMat} - \frac{\tau}{2}\bigfnorm{\subGMat - \subGMat^k}^2,~\forall\subGMat\in \{\subGMat\mid \subGMat_{ij}\in \{0,\pm 1\}  \}. 
	\end{align*}
	In particular, 
	\[
	\innerprod{\dataMatTrans\projMat^{k+1}}{\subGMat^{k+1}} - \frac{\tau}{2}\bigfnormsquare{\subGMat^{k+1}-\subgradMat^k} \geq \innerprod{\dataMatTrans\projMat^{k+1}}{\subGMat^k},
	\]
	i.e.,
	\begin{align}\label{eq:proof_finite_conv_0}
	F(\projMat^{k+1},\subGMat^{k+1}) - F(\projMat^{k+1},\subGMat^k) \geq \frac{\tau}{2}\bigfnormsquare{\subGMat^{k+1}-\subGMat^k}.
	\end{align}
	Since $F(\projMat^{k+1},\subGMat^k)\geq F(\projMat^k,\subGMat^k)$, we can use $F(\projMat^k,\subGMat^k)$ to replace $F(\projMat^{k+1},\subGMat^k)$ above.  
	Summing it from $k=0$ to $\infty$ gives
	\begin{align*}
	\frac{\tau}{2}\sum^{\infty}_{k=0}\bigfnormsquare{\subGMat^{k+1}-\subGMat^k} &\leq \lim_{k\rightarrow\infty}F(\projMat^{k},\subGMat^{k}) - F(\projMat^0,\subGMat^0)\\
	&\leq F^{\max} <+\infty.
	\end{align*}
	This together with Proposition \ref{prop:sk_sk1_geq2} shows that   there exists a $k_0$, such that     $\subGMat^{k_0} = \subGMat^{k_0+1}=\cdots$.
	
	Item 2: Let $k^\prime$ be the first step that $\subGMat^{k^\prime} = \subGMat^{k^\prime+1}$. Since with Setting \ref{set:1}, 
	$\projMat^{k+1}$ is uniquely determined by $\dataMat\subGMat^k$,   we get
	\begin{align}\label{eq:proof_finite_conv_1}
	\projMat^{k^\prime + 1} = \operatorname{PD}\bigxiaokuohao{\dataMat\subGMat^{k^\prime}} = \operatorname{PD}\bigxiaokuohao{\dataMat\subGMat^{k^\prime+1}} = \projMat^{k^\prime+2},
	\end{align}
	and consequently,
	\begin{align*}
	\subGMat^{k^\prime+2} &= \operatorname{sgn}\bigxiaokuohao{\tau\subGMat^{k^\prime+1} + \dataMatTrans\projMat^{k^\prime+2}} \\
	&= \operatorname{sgn}\bigxiaokuohao{\tau\subGMat^{k^\prime} + \dataMatTrans\projMat^{k^\prime+1}} = \subGMat^{k^\prime+1}.
	\end{align*}
	Continuing the procedure we always have $\subGMat^{k^\prime}=\subGMat^{k^\prime+1}=\subGMat^{k^\prime+2}=\cdots$ and $\projMat^{k^\prime}=\projMat^{k^\prime+1}=\projMat^{k^\prime+2}=\cdots$. Without loss of generality we may assume that $k^\prime$ is the first $k$ such that $\subGMat^k = \subGMat^{k+1}$.   On the other hand, the definition of $k^\prime$ in connection with Proposition \ref{prop:sk_sk1_geq2} shows that   
	\begin{align}\label{eq:proof_finite_conv_2}
	\begin{split}
	F(\projMat^{k+1},\subGMat^{k+1}) - F(\projMat^k,\subGMat^k) &\geq \frac{\tau}{2}\bigfnormsquare{\subGMat^{k+1}-\subGMat^k} \\
	&\geq \frac{\tau}{2},~k=0,1,\ldots,k^\prime - 1,
	\end{split}
	\end{align}
	and 
	\[
	F(\projMat^{k+1},\subGMat^{k+1})=F(\projMat^k,\subGMat^k),~k=k^\prime,k^\prime+1,\ldots,
	\]
	i.e., the objective value strictly increases with the quantity at least $\frac{\tau}{2}$ in each step in the first $k^\prime-1$ iterations, and is fixed after $k=k^\prime$. As a result, summing \eqref{eq:proof_finite_conv_2} from $k=0$ to $k^\prime-1$ yields
	\begin{align*}
	k^\prime\cdot \frac{\tau}{2}&\leq \lim_{k\rightarrow\infty}F(\projMat^{k},\subGMat^{k})-F(\projMat^0,\subGMat^0) \leq F^{\max}\\
	& ~\Rightarrow~ k^\prime \leq \left\lceil \frac{2F^{\max}}{\tau}\right\rceil,
	\end{align*}
	namely, the algorithm stops at most $\left\lceil \frac{2F^{\max}}{\tau}\right\rceil$ steps. 
	
	Item 3:  
	Denote 
	\[\tau_1:= \min \bigdakuohao{  |\mathbf x^{\top}_i\mathbf u_j^{k^\prime+1}  | ~\mid~ \mathbf x^\top_i\mathbf u_j^{k^\prime+1}\neq 0 }>0.\]  Now,
	if the parameter $\tau$ is chosen such that $0<\tau<\tau_1$, then whenever $\mathbf x^{\top}_i\mathbf u^{k^\prime+1}_j\neq 0$, $\tau< \tau_1\leq |\mathbf x^{\top}_i\mathbf u^{k^\prime+1}_j|$, and so
	\begin{align*}
	\subGMat^{k^\prime+1}_{ij} &=\operatorname{sgn}\bigxiaokuohao{\tau\subGMat^{k^\prime}_{ij} + \mathbf x^{\top}_i\mathbf u_j^{k^\prime+1}}\\
	& \left\{\begin{array}{ll} =\operatorname{sgn}\bigxiaokuohao{\mathbf x^{\top}_i\mathbf u^{k^\prime+1}_j},&{\rm if}~\mathbf x^{\top}_i\mathbf u^{k^\prime+1}_j\neq 0,\\
	\in [-1,1],&{\rm if}~ \mathbf x^{\top}_i\mathbf u^{k^\prime+1}_j = 0,
	\end{array}\right.
	\end{align*}
	namely, 
	\[
	\subgradMat^{k^\prime+1} \in \partial \bigonenorm{ V}\Big |_{V = \dataMatTrans\projMat^{k^\prime+1} },
	\]
	which together with $\projMat^{k^\prime+1} = \operatorname{PD}\bigxiaokuohao{\dataMat\subGMat^{k^\prime+1}}$ in \eqref{eq:proof_finite_conv_1} shows that $\projMat^{k^\prime+1}$ is a FOC point \eqref{eq:FOC_fpf}. It is also clear that such a point also satisfies $\subGMat^{k^\prime+1} \in \arg\max_{ \{ \subGMat\mid \subGMat_{ij}\in \{0,\pm 1 \} \}}\innerprod{\dataMatTrans\projMat^{k^\prime+1}}{\subGMat}$ and $\projMat^{k^\prime+1} \in \arg\max_{\projMat\in\st{d,K}}\innerprod{\dataMat\subGMat^{k^\prime+1}}{\projMat},$
	which is a partial maximizer of \eqref{prob:l1_pca_bilinear}. 
\end{proof}

\begin{remark}
	\label{rmk:no_prox_u_subproblem}
	It is also possible to impose a proximal term to the $\projMat$-subproblem, resulting into the computation $\projMat^{k+1} = \operatorname{PD}\bigxiaokuohao{\tau \projMat^k + \dataMat\subGMat^k }$. The question is now that $\projMat^{k+1}$ is not only determined by $\dataMat\subGMat^{k+1}$ but also   by $\projMat^k$. Since $\projMat^k\in \st{d,K}$ which is not discrete,  it is unclear whether the number of possible $\projMat^k$ is finite in this situation, and the finite-step convergence analysis may not go through. 
\end{remark}

\section{PAM${\rm e}$ and Finite-Step Convergence} \label{sec:am_convergence}
In \cite{wang2021linear}, the authors proposed a proximal alternating minimization method with an additional extrapolated step (PAMe) for $L_1$-norm PCA. By showing that the Kurdyka-\L{}ojasiewicz exponent of the problem is $1/2$, the authors were able to prove that PAMe converges globally and linearly, and the output is   a critical point if a certain parameter condition is met. Moreover, it was observed that PAMe is more efficient than NGA. Therefore, it would be also interesting to investigate the finite-step property of PAMe. First we recall PAMe of \cite{wang2021linear}   with our notations: 
\begin{align}
\label{alg:pame_org}
\begin{split}
&  {\rm 1.~Compute}~ \subGMat^{k+1} \in \arg\max_{ \{ \subGMat\mid \subGMat_{ij}\in \{0,\pm 1 \} \} }\innerprod{\dataMatTrans E^{k }}{\subGMat} \\
&~~~~~~~~~~~~~~~~~~~~~~~~~~~~~- \frac{\tau}{2}\bigfnormsquare{\subGMat - \subGMat^{k}},\\
& {\rm 2. ~Compute}~\projMat^{k+1} \in \arg\max_{\projMat\in\st{d,K}}\innerprod{\dataMat\subGMat^{k+1}}{\projMat}  \\
&~~~~~~~~~~~~~~~~~~~~~~~~~~~~~- \frac{\beta}{2}\bigfnormsquare{\projMat - \projMat^k},\\
& {\rm 3. ~Update}~E^{k+1} =  \projMat^{k+1} -  \gamma \bigxiaokuohao{\projMat^k-\projMat^{k+1}  },
\end{split}
\end{align} 
where $\tau>0,\beta>0$, and $\gamma\in [0,1)$. When $\tau=\beta=\gamma=0$, it reduces to NGA. Note that  the parameters in \cite{wang2021linear} can vary every iterations, while to keep things simple, we fix them. The constraint of $\subGMat$ is $\subGMat_{ij}\in \{\pm 1\}$ in \cite{wang2021linear}, while to be consistent, we still let $\subGMat_{ij}\in \{ 0,\pm 1 \}$, and this does not affect the results. The iterative scheme of PAMe is then written as follows:
\begin{center}
	\begin{boxedminipage}{8.75cm}
		\begin{align}
		\label{alg:pame}
		\begin{split}
		&\subGMat^{k+1} = \operatorname{sgn}\bigxiaokuohao{\tau \subGMat^k + \dataMatTrans E^k}, \\
		{\rm (} {\rm PAMe)}~~&\projMat^{k+1} = \operatorname{PD}\bigxiaokuohao{\beta \projMat^k+\dataMat\subGMat^{k+1}}, \\
		&E^{k+1} =  \projMat^{k+1} -  \gamma \bigxiaokuohao{\projMat^k-\projMat^{k+1}  }.
		\end{split}
		\end{align}
	\end{boxedminipage}
\end{center}

\subsection{Finite-step convergence of PAMe}
At a first glance, since the $\subGMat$-step takes the form $\subGMat^{k+1} = \operatorname{sgn}\bigxiaokuohao{\tau \subGMat^k + \dataMatTrans E^k}$, it is expected that as that of $\subGMat$-PNGA, sufficient increasing inequality \eqref{eq:proof_finite_conv_0} can be established and so $\subGMat^k$ will be fixed after finitely many steps. Indeed, this is true if the parameters satisfy certain assumptions.
\begin{proposition}
	\label{prop:pame_s_stops}
	Let $\{\projMat^k,\subGMat^k  \}$ be generated by PAMe \eqref{alg:pame}. Then there exists a small enough   $\gamma$, such that after finitely many steps, $\subGMat^k=\subGMat^{k+1}=\subGMat^{k+2}=\cdots$. 
\end{proposition}
\begin{proof}
	The proof is similar to those of \cite[Theorem 2]{wang2021linear} and Theorem \ref{th:finite_step_conv}.  First  by the definition of $\subGMat^{k+1}$, we have
	\begin{align}
	\label{eq:proof_finite_conv_pame_1}
	\innerprod{\dataMatTrans E^k}{\subGMat^{k+1}} - \innerprod{\dataMatTrans E^k}{\subGMat^k} \geq \frac{\tau}{2}\bigfnormsquare{\subGMat^{k+1}-\subGMat^k}.
	\end{align}
	\eqref{eq:proof_finite_conv_pame_1} together with the definition of $E^k$  means that
	\begin{align}
	\label{eq:proof_finite_conv_pame_2}
	\begin{split}
	&	\innerprod{\dataMatTrans\projMat^{k}}{\subGMat^{k+1}} - \innerprod{\dataMatTrans\projMat^k}{\subGMat^k} \\
	\geq &  \frac{\tau}{2}\bigfnormsquare{\subGMat^{k+1}-\subGMat^k} + \gamma\innerprod{\dataMatTrans\bigxiaokuohao{\projMat^{k}-\projMat^{k-1}}}{\subGMat^{k+1}-\subGMat^k} \\
	\geq &\frac{\tau}{2}\bigfnormsquare{\subGMat^{k+1}-\subGMat^k} - \frac{\gamma\tau}{2}\bigfnormsquare{\subGMat^{k+1}-\subGMat^k} \\
	&~~-  \frac{\gamma\|\dataMat\|_2^2}{2\tau}\bigfnormsquare{\projMat^k - \projMat^{k-1}}.
	\end{split}
	\end{align}
	The definition of $\projMat^{k}$ shows that
	\begin{align}
	\label{eq:proof_finite_conv_pame_3}
	\innerprod{\projMat^{k}}{\dataMat\subGMat^{k}} - \innerprod{\projMat^{k-1}}{\dataMat\subGMat^{k}} \geq \frac{\beta}{2}\bigfnormsquare{\projMat^{k}-\projMat^{k-1}}.
	\end{align}
	Denote $\eta := \frac{\gamma\|\dataMat\|_2^2}{\tau}$. Combining \eqref{eq:proof_finite_conv_pame_2} and \eqref{eq:proof_finite_conv_pame_3} together, we get
	\begin{align*}
	&		\innerprod{\dataMatTrans\projMat^{k}}{\subGMat^{k+1}} -  \innerprod{\dataMatTrans\projMat^{k-1}}{ \subGMat^{k}} \\
	\geq&  \frac{\tau(1-\gamma)}{2}\bigfnormsquare{\subGMat^{k+1}-\subGMat^k} + \frac{\beta-\eta}{2}\bigfnormsquare{\projMat^{k} - \projMat^{k-1}}. 
	\end{align*}
	Summing the above inequality from $k=1$ to infinity, we have
	\begin{align*}
	&\sum^\infty_{k=1}\bigxiaokuohao{\frac{\tau(1-\gamma)}{2}\bigfnormsquare{\subGMat^{k+1}-\subGMat^k} + \frac{\beta-\eta}{2}\bigfnormsquare{\projMat^{k} - \projMat^{k-1}}} \\
	\leq & \lim_{k\rightarrow\infty} \innerprod{\dataMatTrans\projMat^{k}}{\subGMat^{k+1}} -  \innerprod{\dataMatTrans\projMat^{0}}{ \subGMat^{1}}  \\ 
	\leq & \max_{\projMat \in\st{d,K}, { \{ \subGMat\mid \subGMat_{ij}\in \{0,\pm 1 \} \} }  } \innerprod{\dataMatTrans\projMat}{\subGMat}       =  F^{\max}.
	\end{align*}
	Now, if $\gamma$ is chosen sufficiently small such that $\beta \geq \eta$, i.e., $\gamma \leq \frac{ \beta \tau }{ \|\dataMat\|_2^2 }$, the above gives
	\begin{align*}
	\sum^\infty_{k=0}\frac{\tau(1-\gamma)}{2}\bigfnormsquare{\subGMat^{k+1}-\subGMat^k}\leq  F^{\max}.
	\end{align*}
	Further by setting $\gamma < \min \{1,\frac{ \beta \tau }{ \|\dataMat\|_2^2 } \}$, the above relation 
	together with Proposition \ref{prop:sk_sk1_geq2} shows that  after finitely many steps, $\subGMat^k=\subGMat^{k+1}=\subGMat^{k+2}=\cdots$. 
\end{proof}
\begin{remark}
	Unlike Theorem \ref{th:finite_step_conv}, it may not be possible to derive an upper bound on the number of steps. This is due to the update of $\projMat^{k+1}$ does not ensure that if $\subGMat^k=\subGMat^{k+1}$ for some $k$, then there must hold that $\subGMat^k=\subGMat^{k+1}=\subGMat^{k+2}=\cdots$; c.f. the proof of item 2 of Theorem \ref{th:finite_step_conv}. In other words, the objective value may not decrease at least the quantity of $\tau(1-\gamma)/2$ in each iteration. 
\end{remark}

Denote $\subGMat^*$ as the stopping point of $\{ \subGMat^k \}$, i.e., 
$ \subGMat^{k^\prime}=\subGMat^{k^\prime+1}=\cdots = \subGMat^*$,
where $\subGMat^{k^\prime}$ is the first step that $\subGMat^k$ will not change any more. Thus the first line of PAMe \eqref{alg:pame} reduces to:
\begin{align}
\label{eq:s_pame_fixed}
\subGMat^* = \operatorname{sgn}\bigxiaokuohao{ \tau\subGMat^* + \dataMatTrans E^{k} },~k=k^\prime,k^\prime+1,\ldots.
\end{align}
Therefore,   there is no need to excute the first and the third lines of PAMe anymore when $k\geq k^\prime$, and the algorithm reduces to the following single-line scheme:
\begin{align}\label{eq:sec:pame:1}
\projMat^{k+1} = \operatorname{PD}\bigxiaokuohao{\beta \projMat^k+\dataMat\subGMat^{*}},~k =k^\prime ,k^\prime+1\ldots. 
\end{align} 
As explained in Remark \ref{rmk:no_prox_u_subproblem}, the additional $\beta\projMat^k$ in PD above might lead to infinitely many $\projMat^{k+1}$.  To see it clearer, consider $n=1$, and we set $\beta=1$ and $\|\dataMat\subGMat^*\|=1$ for simplicity. Now the update of $\projMat^{k+1}$ reduces to 
\[\projMat^{k+1}= (\projMat^k + \dataMat\subGMat^*)/\|\projMat^k+\dataMat\subGMat^*\|,\]
which means that the angle between $\dataMat\subGMat^*$ and $\projMat^{k+1}$ is half of that between $\dataMat\subGMat^*$ and $\projMat^k$, but the angle cannot be exactly zero even if $k\rightarrow\infty$. Therefore, there is no hope to deduce the finite-step convergence on $\projMat^k$. 

Nevertheless, using only the finite-step convergence property on $\subGMat$, we can still establish the finite-step convergence on PAMe to a certain extent. To this end, observe that
\eqref{eq:sec:pame:1}  is exactly a proximal point algorithm (see, e.g., \cite{parikh2014proximal}) for solving $\max_{\projMat \in \st{d,K}}\innerprod{\dataMat\subGMat^*}{\projMat}$, i.e., finding a PD of $\dataMat\subGMat^*$. Similar to \cite{wang2021linear},  the whole sequence $\{\projMat^k \}$ can be shown to converge to a single limit point, whose proof will be omitted:
\begin{proposition}
	\label{prop:global_conv_u_pame}
	$\{\projMat^k \}$ of \eqref{eq:sec:pame:1} converges to a point $\projMat^*\in \st{d,K}$, i.e., $\lim_{k\rightarrow\infty}\projMat^k=\projMat^*$; moreover, 
	\begin{align} \label{eq:u_pame_fixed}
	\projMat^*=\operatorname{PD}(\beta \projMat^* + \dataMat\subGMat^*).
	\end{align}
\end{proposition}

It follows from \eqref{eq:u_pame_fixed} and the definition of PD that there exists an  $H\in \mathbb S^{n\times n}_+$, such that $\beta\projMat^* + \dataMat\subGMat^* = \projMat^* H$, i.e., 
\begin{align}
\label{eq:sec:pame:2}
\dataMat\subGMat^* = \projMat^*\Lambda,~{\rm where}~\Lambda:= H-\beta I\in\mathbb S^{n\times n}.
\end{align}

On the other hand, by the definition of $E^k$ in PAMe, $\lim_{k\rightarrow\infty}E^k = \projMat^*$. Letting $k\rightarrow\infty$ in \eqref{eq:s_pame_fixed} yields $\subGMat^* = \operatorname{sgn}\bigxiaokuohao{ \tau\subGMat^* + \dataMatTrans \projMat^* }$. If $\tau$ is taken small enough, namely, $\tau < \min \{ |\mathbf x^\top_i\mathbf u^*_j| \mid \mathbf x^\top_i\mathbf u^*_j\neq 0  \} $, then 
\begin{align}\label{eq:sec:pame:3}
\subGMat^*\in \partial \bigonenorm{V}\Big |_{V = \dataMatTrans\projMat^* }.
\end{align}

In view of the definition of the KKT point \eqref{eq:KKT}, \eqref{eq:sec:pame:3} and \eqref{eq:sec:pame:2} implies that there must exist a $\projMat$-factor of $\dataMat\subGMat^*$, termed as $\projMat^*$, such that $\projMat^*$ is a KKT point of the form \eqref{eq:KKT}.  

In view of \eqref{eq:u_pame_fixed}, if $\beta$ is also chosen small enough such that $\beta < \sigma^+_{\min}(\dataMat\subGMat^*)$, i.e., $\beta$ is smaller than the smallest positive singular value of $\dataMat\subGMat^*$, then it follows from Lemma \ref{prop:relation:2} that
\begin{align}\label{eq:sec:pame:4}
\projMat^* = \operatorname{PD}(\dataMat\subGMat^*).
\end{align}
\eqref{eq:sec:pame:3} and \eqref{eq:sec:pame:4} implies that there must exist a $\projMat$-factor of $\dataMat\subGMat^*$, termed as $\projMat^*$, such that $\projMat^*$ is a FOC point of the form \eqref{eq:FOC_fpf}.  

In particular, if ${\rm rank}(\dataMat\subGMat^*)=K$, then Lemma \ref{lem:polar_max_sufficient} ensures that $\projMat^*$ is unique, i.e., in finitely many steps, PAMe can output the point $ \operatorname{PD}(\dataMat\subGMat^{k^\prime})$ which 
is a FOC point of the form \eqref{eq:FOC_fpf}.

Based on the above discussions, we conclude that:
\begin{theorem}[Finite-step convergence of PAMe \eqref{alg:pame}]
	\label{th:finite_conv_pame}
	Let $\{\projMat^k,\subGMat^k  \}$ be generated by PAMe \eqref{alg:pame}. The following statements hold:
	\begin{enumerate}
		\item If $\gamma$ is small enough, then after finitely many steps, $\subGMat^k$ will not change, i.e., $ \subGMat^{k^\prime}=\subGMat^{k^\prime+1}=\cdots =: \subGMat^* $, 
		where	 $k^\prime$ denotes the first index that $\subGMat^k$ will not change anymore;
		\item If $\tau$ is also small enough, then there must exist a $\projMat$-factor of $\dataMat\subGMat^*$, termed as $\projMat^*$, such that $\projMat^*$ is a KKT point of the form \eqref{eq:KKT};
		\item  If furthermore, $\beta$ is also small enough, then   $\projMat^*$ above is a FOC point of the form \eqref{eq:FOC_fpf};
		\item If ${\rm rank}(\dataMat\subGMat^*)=K$, then in finitely many steps, PAMe     can stop at the $k^\prime$ step and output the point $ \operatorname{PD}(\dataMat\subGMat^{*})$, which 
		is a FOC point of the form \eqref{eq:FOC_fpf}.
	\end{enumerate}
\end{theorem}

{\bf $K=1$ case} It is clear that when $K=1$, the rank assumption in Theorem \ref{th:finite_conv_pame} can be removed, and the results  are presented as follows, where we still use $\mathbf u$ and $\mathbf s$ instead of $\projMat$ and $\subGMat$ as they are now vectors. 
\begin{theorem}[Finite-step convergence of PAMe \eqref{alg:pame} when $K=1$]
	\label{th:finite_conv_pame_K1}
	When $K=1$ and let $\{\mathbf u^k,\mathbf s^k \}$ be generated by PAMe \eqref{alg:pame} where $\mathbf u^0$ is chosen such that $\bigonenorm{\dataMatTrans\mathbf u^0}\neq 0$. If $\gamma,\tau,\beta$ are small enough, then after finitely many steps, $\mathbf s^k$ will not change, and $\dataMat\mathbf s^k/\|\dataMat\mathbf s^k\|$ is a FOC point of the form \eqref{eq:FOC_fpf}. 
\end{theorem}

In the next subsection, we will consider removing the proximal term in the $\projMat$-subproblem and study the finite-step convergence. 

\subsection{Convergence results of $\subGMat$-PAMe}
Since it is hard to obtain the finite-step convergence on $\projMat^k$ due to the presence of $\beta\projMat^k$ in the $\projMat$-step in PAMe,  in this subsection we simply set $\beta=0$, leading to the following iterative scheme:
\begin{center}
	\begin{boxedminipage}{8.75cm}
		\begin{align}
		\label{alg:pame_beta0}
		\begin{split}
		&\subGMat^{k+1} = \operatorname{sgn}\bigxiaokuohao{\tau \subGMat^k + \dataMatTrans E^k}, \\
		{\rm (}\subGMat{\rm -PAMe)}~~&\projMat^{k+1} = \operatorname{PD}\bigxiaokuohao{\dataMat\subGMat^{k+1}}, \\
		&E^{k+1} =  \projMat^{k+1} -  \gamma \bigxiaokuohao{\projMat^k-\projMat^{k+1}  },
		\end{split}
		\end{align}
	\end{boxedminipage}
\end{center}

\noindent 
As the proximal term is only imposed on the $\subGMat$-subproblem, similar to \eqref{alg:am_prox_l1_pca}, \eqref{alg:pame_beta0} will be termed as $\subGMat$-PAMe in the sequel. 

Without the proximal term on the $\projMat$-subproblem, it is hard to obtain the sufficiently ascending inequality on $\projMat$, such as \eqref{eq:proof_finite_conv_pame_3}, so as to establish the finite-step convergence on $\subGMat$. For this reason we have to resort to the full rank assumption on ${\rm rank}(\dataMat\subGMat^k)$, similar to Theorem \ref{th:global_conv}. 
We have the following:
\begin{theorem}[Finite-step convergence for $\subGMat$-PAMe \eqref{alg:pame_beta0}] \label{th:finite_conv_pame0}
	Let $\{\projMat^k,\subGMat^k\}$ be generated by $\subGMat$-PAMe \eqref{alg:pame_beta0}. Assume that $\rank{\dataMat\subGMat^k} = K$ for all $k$. The following two statements hold:
	\begin{enumerate}
		\item  If $\gamma$ is small enough, such that    after at most $\left\lceil \frac{8F^{\max}}{\tau(1-\gamma)}\right\rceil$ steps, $\projMat^k=\projMat^{k+1}=\cdots$ and $\subGMat^k=\subGMat^{k+1}=\cdots$; 
		\item In addition, there exists a $\tau_0>0$, such that when $0<\tau<\tau_0$,   $\projMat^{k^\prime}$ is a FOC point of the form \eqref{eq:FOC_fpf}, where $k^\prime$ is the first step that $\subGMat^{k}$ does not change any more. 
	\end{enumerate}   
\end{theorem}
\begin{proof}	
	Denote $\dataMat\subGMat^k = \projMat^{k}H^{k}$ as the PD of $\dataMat\subGMat^k$, where $H^k\in\PSDD{K\times K}$ and $\projMat^k$ are both uniquely determined due to $\rank{\dataMat\subGMat^k} = K$. The number of possible $H^k$'s is   finite due to the finiteness of $\subGMat^k$. Thus   define $\lambda = \min_{k} \lambda_{\min}(H^k)>0$.   It follows from Lemma \ref{lem:polar_max_sufficient}  that
	\begin{align}
	\label{eq:proof_finite_conv_pame_31}
	\innerprod{\projMat^{k}}{\dataMat\subGMat^{k}} - \innerprod{\projMat^{k-1}}{\dataMat\subGMat^{k}} \geq \frac{\lambda}{2}\bigfnormsquare{\projMat^{k}-\projMat^{k-1}}.
	\end{align}
	Note that \eqref{eq:proof_finite_conv_pame_2}  in Proposition \ref{prop:pame_s_stops} still holds in the context. 
	Combining \eqref{eq:proof_finite_conv_pame_2} and \eqref{eq:proof_finite_conv_pame_31} yields 
	\begin{align*}
	&		\innerprod{\dataMatTrans\projMat^{k}}{\subGMat^{k+1}} -  \innerprod{\dataMatTrans\projMat^{k-1}}{ \subGMat^{k}} \\
	\geq&  \frac{\tau(1-\gamma)}{2}\bigfnormsquare{\subGMat^{k+1}-\subGMat^k} + \frac{\lambda-\eta}{2}\bigfnormsquare{\projMat^{k} - \projMat^{k-1}},
	\end{align*}
	where  we  recall that $\eta = \frac{\gamma\|\dataMat\|_2^2}{\tau}$.
	Taking $\gamma < \min \{ 1,  \frac{ \lambda \tau }{ \|\dataMat\|_2^2} \}$ we also have 
	\begin{align}\label{eq:sec:pame0:1}
	\sum^\infty_{k=0}\frac{\tau(1-\gamma)}{2}\bigfnormsquare{\subGMat^{k+1}-\subGMat^k}\leq  F^{\max},
	\end{align}
	i.e., $\subGMat^k$ will not change in finitely many steps. 
	
	Let $k^\prime$ be the first step that $\subGMat^{k^\prime} = \subGMat^{k^\prime+1}=\subGMat^{k^\prime+2}$, i.e., the first step such that three successive $\subGMat^k$ are equal to each other. We now claim that $\subGMat^k=\subGMat^{k+1}=\cdots$ and $\projMat^k=\projMat^{k+1}=\cdots$ when $k\geq k^\prime$. To see this, 
	it follows from $\projMat^{k}=\operatorname{PD}(\dataMat\subGMat^k)$ that $\projMat^{k^\prime}=\projMat^{k^\prime+1}=\projMat^{k^\prime+2}$, and so $E^{k^\prime+1} = \projMat^{k^\prime+1} - \gamma(\projMat^{k^\prime}-\projMat^{k^\prime+1}) = \projMat^{k^\prime+1}$, and similarly, $E^{k^\prime+2} = \projMat^{k^\prime+2}$; thus $E^{k^\prime+1}=E^{k^\prime+2}$.  As a consequence, $\subGMat^{k^\prime+3}={\rm sgn}(\tau \subGMat^{k^\prime+2}+ \dataMat E^{k^\prime+2}) = {\rm sgn}(\tau\subGMat^{k^\prime+1} + \dataMat E^{k^\prime+1}) = \subGMat^{k^\prime+2}$, and hence $\projMat^{k^\prime+3}=\projMat^{k^\prime+2}$. Continuing this vein,  we have $\subGMat^k=\subGMat^{k+1}=\cdots$ and $\projMat^k=\projMat^{k+1}=\cdots$ when $k\geq k^\prime$. 
	
	We now upper bound the number of steps. The above discussions show that before $\subGMat^k$ stops, it must change whithin every three steps, i.e., 
	\[
	\bigfnorm{\subGMat^{k+2}-\subGMat^{k+1}}^2 + \bigfnorm{\subGMat^{k+1}-\subGMat^k}^2 \geq 1,\forall k < k^\prime,
	\]
	which together with \eqref{eq:sec:pame0:1} implies that $k^\prime \leq \left\lceil \frac{4F^{\max}}{\tau(1-\gamma)}\right\rceil$. 
	
	Denote 
	\[\tau_1:= \min \bigdakuohao{  |\mathbf x^{\top}_i\mathbf u_j^{k^\prime+1}  | ~\mid~ \mathbf x^\top_i\mathbf u_j^{k^\prime+1}\neq 0 }>0.\]
	If $\tau<\tau_1$, then $\subGMat^{k^\prime+2} = {\rm sgn}(\tau\subGMat^{k^\prime+1} + \dataMatTrans E^{k^\prime+1}) = {\rm sgn}(\tau \subGMat^{k^\prime+1}+\dataMatTrans\projMat^{k^\prime+1})$  gives that
	\[
	\subGMat^{k^\prime+2} \in \partial \bigonenorm{ V} \Big |_{V = \dataMatTrans\projMat^{k^\prime+1}};
	\]
	on the other hand, $\projMat^{k^\prime+1} = \operatorname{PD}(\dataMat\subGMat^{k^\prime+1})$. These together with $\subGMat^{k^\prime+1}=\subGMat^{k^\prime+2}$ show that $\projMat^{k^\prime+1}$ is a FOC point of the form \eqref{eq:FOC_fpf}.
\end{proof}
\begin{remark}
	If the $\operatorname{sgn}(\cdot)$ function only takes $\pm 1$ as that in \cite{wang2021linear}, then the upper bound can be improves slightly to $\left\lceil\frac{ F^{\max}}{\tau(1-\gamma)}\right\rceil$.
\end{remark}

{\bf $K=1$ case} When $K=1$,  the full column rank assumption can also be removed. 
\begin{theorem}[Finite-step convergence for $\subGMat$-PAMe \eqref{alg:pame_beta0} when $K=1$]
	\label{col:finite_conv_pame_K_1}
	Consider the case that $K=1$. Let $\{ \mathbf u^k,\mathbf s^k \}$ be generated by $\subGMat$-PAMe \eqref{alg:pame_beta0} where $\mathbf u^0$ is chosen such that $\bigonenorm{\dataMatTrans\mathbf u^0}\neq 0$. Then there exists    small enough parameters $\tau$ and $\gamma$, such that    after at most $\left\lceil \frac{4F^{\max}}{\tau(1-\gamma)}\right\rceil$ steps, the algorithm stops at a FOC point of the form \eqref{eq:FOC_fpf}.
\end{theorem}
\begin{proof}
	The proof follows from those of Theorems \ref{th:global_conv_K_1} and \ref{th:finite_conv_pame0} and is omitted.
\end{proof}

\section{Conclusions}\label{sec:conclusions}

The non-greedy algorithm (NGA) \cite{nie2011robust} and PAMe of \cite{wang2021linear}  for $L_1$-norm PCA were     studied. NGA was first treated as a conditional subgradient and its finite-step convergence to a FOC point was presented under a full rank assumption; such an assumption was removed when $K=1$. NGA was then considered as an alternating maximization method and it was shown that the objective value does not change after at most $\left\lceil\frac{ F^{\max}}{\tau_0}\right\rceil$  steps. The stopping point is a FOC point. 
To enhance the convergence,   NGA with one proximal term imposed on the sign variable ($\subGMat$-PNGA) was then studied,  whose iterative points must stop after at most $\left\lceil\frac{2F^{\max}}{\tau}\right\rceil$ steps. The stopping point is    a FOC point if $\tau$ is small enough. 

For PAMe, if $\gamma$ is small enough, then the sign variable will not change after finitely many steops; if in addition, $\beta,\tau$    are also small enough and a full rank assumption is satisfied, then PAMe can output a FOC point. Moreover, if there is no proximal term imposed on the projection matrix related subproblem in PAMe ($\subGMat$-PAMe), then after at most   $\left\lceil \frac{4F^{\max}}{\tau(1-\gamma)}\right\rceil$ steps, the iterative point will not change anymore and  the stopping point is also a FOC point, provided similar assumptions as those for PAMe. The full rank assumption can be removed in the $K=1$ case.

Although PAMe is more efficient than NGA in practice, as observed in \cite{wang2021linear}, the finite-step convergence of PAMe established in this work is not as good as that for NGA. Therefore, there is still a gap between theory and practice that still needs further research. 
On the other hand,
It might be possible to devise similar analysis for related methods such as $L_1$-HOOI for $L_1$-norm Tucker decomposition \cite{chachlakis2019l1,chachlakis2020l1}, for the non-greedy algorithm for $L_{21}$-norm PCA \cite{nie2021nongreedy}, and  PALMe for rotational invariant $L_1$-norm PCA    \cite{zheng2022linearly}.

 {\scriptsize\section*{Acknowledgement}  This work was supported by the National Natural Science Foundation of China (Grant No. 12171105),  the    Fok Ying Tong Education Foundation (Grant No. 171094), and the special foundation for Guangxi Ba Gui Scholars. }

   \bibliography{../tensor,../TensorCompletion,../orth_tensor,../l1-PCA_related,../cg_fw}
  \bibliographystyle{plain}

   \appendix

\end{document}